\documentclass[11pt]
{amsart}
\usepackage{amssymb,amsmath,amsthm,amsfonts,amsopn,url,color,%enumerate,
mathtools,microtype,MnSymbol,%dutchcal,
mathrsfs, tikz, tikz-cd}
\usepackage[pagebackref]{hyperref}
\usepackage{scalerel}
\usepackage{stackengine,wasysym}
\usepackage[shortlabels]{enumitem}

\usepackage[normalem]{ulem}
\usepackage[all]{xy}
\input xy
\xyoption{all}
\usepackage{amscd}
\usepackage{soul}

\theoremstyle{plain}
\newtheorem{thm}{Theorem}[section]

\newtheorem{prop}[thm]{Proposition}
\newtheorem{lemma}[thm]{Lemma}
\newtheorem{cor}[thm]{Corollary}

\theoremstyle{definition}
\newtheorem{defn}[thm]{Definition}
\newtheorem*{defn*}{Definition}
\newtheorem*{question*}{Question}
\newtheorem{question}{Question}

\newtheorem*{example*}{Example}
\newtheorem{rem}[thm]{Remark}
\newtheorem*{rem*}{Remark}

\newtheorem{nota}[thm]{Notation}
\newtheorem*{nota*}{Notation}

\newcommand{\field}[1]{\mathbb{#1}}
\newcommand{\N}{\field{N}}
\newcommand{\Z}{\field{Z}}
\newcommand{\Q}{\field{Q}}

\newcommand{\A}{\field{A}}

\newcommand{\PP}{\field{P}}
\newcommand{\ideal}[1]{\mathfrak{#1}}
\newcommand{\m}{\ideal{m}}
\newcommand{\n}{\ideal{n}}
\newcommand{\p}{\ideal{p}}
\newcommand{\q}{\ideal{q}}
\newcommand{\ia}{\ideal{a}}

\newcommand{\func}[1]{\mathrm{#1} \,}

\newcommand{\Spec}{\func{Spec}}

\newcommand{\hgt}{\func{ht}}

\newcommand{\ra}{\rightarrow}

\newcommand{\be}{\begin{enumerate}}
\newcommand{\ee}{\end{enumerate}}

\newcommand{\li}%{{\mathrm{lic}}}
 {\leftfootline}

\newcommand{\onto}{\twoheadrightarrow}

\newcommand{\cO}{\mathcal{O}}

\renewcommand{\phi}{\varphi}
\DeclareMathOperator{\Frac}{Frac}

\DeclareMathOperator{\chr}{char}

\let\int\relax
\DeclareMathOperator{\int}{i}

\DeclareMathOperator{\Aut}{Aut}
\newcommand{\rco}{reciprocal complement}
\newcommand{\rcos}
{reciprocal complements}
\newcommand{\rrec}{ring of reciprocals}

\author{Neil Epstein}
\address{Department of Mathematical Sciences \\ George Mason University \\ Fairfax, VA  22030\\
USA}
\email{nepstei2@gmu.edu}

\author{Lorenzo Guerrieri}
\address{Instytut Matematyki \\
Jagiellonian University \\
30-348 Krak\'ow \\ Poland}
\email{lorenzo.guerrieri@uj.edu.pl}
\thanks{Lorenzo Guerrieri is supported by the grants MAESTRO NCN-UMO-2019/34/A/ST1/00263 -- Research in Commutative Algebra and Representation Theory and NAWA POWROTY-PPN/PPO/2018/1/00013/U/00001 -- Applications of Lie algebras to Commutative Algebra. }

\author{K. Alan Loper}
\address{Department of Mathematics\\
The Ohio State University at Newark\\
Newark, OH 43055\\
USA}
\email{lopera@math.ohio-state.edu}

\title[Reciprocal complement of polynomial rings]{The reciprocal complement of a polynomial ring in several variables over a field}
\date{June 27, 2025}

\begin{document}
\begin{abstract}
The \emph{\rco} $R(D)$ of an integral domain $D$ is the subring of its fraction field generated by the reciprocals of its nonzero elements.  Many properties of $R(D)$ are determined when $D$ is a polynomial ring in $n\geq 2$ variables over a field.  In particular, $R(D)$ is an $n$-dimensional, local, non-Noetherian, non-integrally closed, non-factorial, atomic G-domain, with infinitely many prime ideals at each height other than $0$ and $n$.
\end{abstract}

\maketitle

\section{Introduction}

Let $D$ be an integral domain, with fraction field $F$.  What can we say about the subring $R(D)$ of $F$ generated by the reciprocals of all the nonzero elements of $D$ (called the \emph{\rco}, or \emph{\rrec}, of $D$)?

Simple as the above question is, it appears to be a new one, and as we will see in this paper, the answer can be both surprising and satisfying.  The question arises naturally in the study of \emph{Egyptian domains}, which extends the notion of Egyptian fractions from the integers to arbitrary integral domains.  This study was initiated in \cite{GLO-Egypt} and continued in \cite{nme-Edom}. %\bc initiated with the scope of extending the idea of Egytpian fraction to arbitrary rings \ec
 Recall \cite{GLO-Egypt} that an integral domain $D$ is \emph{Egyptian} if every element of its fraction field $F$ can be written as a sum of (resp. of distinct) reciprocals of elements of $D$ (in general, such an element of $F$ is called \emph{Egyptian} or $D$-\emph{Egyptian}). % \bc the fraction field should be $F$ not $Q$ \ec -- DONE
 From the viewpoint of the above question then, $D$ is Egyptian if and only if $F = R(D)$.  So the distinction between $R(D)$ and $F$ can be seen as a measure of how far an integral domain is from being Egyptian.

The idea of Egyptian domains ultimately comes from the way the ancient Egyptians represented fractions.  Namely, they represented an element of $\Q \cap (0,1)$ as a sum of reciprocals of (distinct) positive integers (so-called \emph{unit fractions}).  More than 8 centuries ago, Fibonacci \cite{DuGr-FibE} showed that this is always possible.  However, it is far from unique.  There are always infinitely many ways to represent a positive rational number as a sum of distinct unit fractions. For the past century or so, number theorists have taken up questions of the diversity of ways to represent fractions as unit fractions.  Indeed, such questions can always be rephrased as diophantine equations. Most prominently, the Erd\H{o}s-Straus conjecture posits that for any $n\geq 5$, $4/n$ can be written as a sum of at most 3 unit fractions.
 
 In addition to its connection with Egyptian fractions, the setting of Egyptian domains and reciprocal complements also isolates natural properties of affine semigroups, rings, and even of algebraic varieties.  For instance, it can distinguish whether a subsemigroup $\Lambda$ of $\Q^n$ is a group, in the following sense. Let $D$ be an Egyptian domain (e.g. $\Z$, or any field).  Then $D[\Lambda]$ is Egyptian if and only if $\Lambda$ is a group (see \cite[Proposition 3]{GLO-Egypt} for ``if'' and \cite[Theorem 2.6]{nme-Edom} for ``only if'').  On the other hand, any \emph{local} domain is Egyptian \cite[Example 3]{GLO-Egypt} and in an affine-local sense, any domain that is finitely generated over a field is \emph{locally} Egyptian \cite[Corollary 3.11]{nme-Edom}, even though $k[x]$ is not Egyptian.  Thus, the Egyptian property is an essentially global property that cannot be checked locally, unlike many ring-theoretic properties.  Passing to algebraic varieties, Dario Spirito has shown \cite[Theorem 2.1]{Spi-recocurve} that when $k$ is an algebraically closed field, and $D$ is a one-dimensional finitely generated $k$-algebra that is a domain, then $D$ is Egyptian if and only if there is some realization $X \subseteq \A^n_k$ of $D$ that is regular at $\infty$ such that $|\bar X \setminus X| \neq 1$, where $\bar X$ is the projective closure of $X$ in $\PP^n_k$.  Otherwise, if $\{p\} = \bar{X} \setminus X$, he shows that the reciprocal complement of $D$ is isomorphic to $\cO_{\bar X, \{p\}}$.

The first named author called an integral domain \emph{Bonaccian} if for any nonzero $f \in F$, either $f$ or $1/f$ can be written as a sum of reciprocals from $D$.  Equivalently, $R(D)$ is a valuation domain. 
 He then showed that a Euclidean domain is always Bonaccian \cite{nme-Euclidean}, and indeed $R(D)$ is either a DVR or a field.
In particular, in \cite{nme-ratfuns}, the first named author showed that the \rco\ of $K[X]$ ($K$ a field, $X$ an indeterminate) is $K[T]_{(T)}$, where $T = 1/X$.

In the current paper, we show that no such thing is true for the \rco\ of a polynomial ring in two or more variables over a field.  Indeed, let $D = K[X_1, \ldots, X_n]$ and $R = R(D)$.  Then $R$ has many properties like that of $D$, but also many interesting, even exotic features. Our main results include the following: \begin{itemize}

    \item Any prime ideal of $R$ is generated by elements of the form $1/f$, $f \in D \setminus K$.  This is a special case of a result that holds in any \rco\ (See Proposition~\ref{pr:gensprimes}).
    \item $R$ is a local ring whose unique maximal ideal is generated by all elements of the form $1/f$ for $f\in D \setminus K$  (See Theorem~\ref{thm:islocalpoly}). This is a special case of a result that holds in any \rco\ (See Theorem~\ref{thm:islocal!!}).  
    \item For every $1\leq i \leq n-1$, $R$ has infinitely many primes of height $i$ (see Theorem~\ref{thm:infprimesallheights}).
    \item $\dim R = n$ (See Theorem~\ref{thm:dim}).
    \item $R$ is atomic (See Theorem~\ref{thm:atomic}).
    \item For every $j \leq n$, there is a prime ideal $\p \in \Spec R$ such that $R_\p$ is isomorphic to the \rco\ of a polynomial ring in $j$ variables over a field (See Proposition~\ref{pr:loclower}).
    \item $R$ is a \emph{G-domain}.  In fact, $R[\prod_{i=1}^n X_i] = \Frac R$  (see Proposition~\ref{pr:Gdom}).
    \item For certain height one primes $\p$, we have that $R_\p$ is a DVR (see Lemma~\ref{Xiprimes}).
%    \item If $n\geq 2$, $R$ is \emph{not} a UFD (See Remark~\ref{rem:notUFD}).
    \item If $n\geq 2$, $R$ is not Noetherian.  In fact it is not even coherent (See Corollary~\ref{cor:notcoh}).
    \item If $n\geq 2$, $R$ is not integrally closed (See Theorem~\ref{thm:notic}).
    \item When $n\leq 2$ and $\hgt \p=1$, $R_\p$ is always Noetherian (see Theorem~\ref{localizations}).
    \item When $n=2$, any finitely generated ideal is contained in all but finitely many prime ideals (See Theorem~\ref{allbutfinitelymnyprimes}).  Since all but two prime ideals of $R$ have height one, this behavior can be seen as an extreme version of Krull's Principal Ideal Theorem. 
\end{itemize}

%\bc maybe recall what is a G-domain and what is the pseudoradical, we could cite the paper on the pseudoradical by Gilmer \ec

A key to our results has been a change in perspective, wherein one uses the $K$-automorphism $\sigma$ of the field $K(X_1, \ldots, X_n)$ that sends $X_i \mapsto 1/X_i$.  Then $R^* := \sigma(R(D))$ contains $D$ as a subring, even though $R^*$ and $R$ are isomorphic as rings.  The effect of this  on reciprocals of explicit polynomials is captured in Lemma~\ref{lem:star}.  We will sometimes use the $R^*$ point of view to analyze $R$, and sometimes the $R$ point of view.

We also use valuations in a variety of ways to control the behavior of prime ideals in $R$.

 Prior to this paper, there were four standard constructions that generally lead to non-Noetherian rings in ways that are essentially different from one another.  These are: \begin{itemize}
    \item Polynomial rings in infinitely many variables over a field or $\Z$ (and their quotients),
    \item Putting a valuation on a field with value group not isomorphic to $\Z$ and extracting the valuation ring,
    \item Pullbacks, and 
    \item Rings of integer-valued polynomials.
\end{itemize}
Now we know there is a fifth such construction: the reciprocal complement.  The fact that the reciprocal complement of such a well-behaved ring as $K[X_1, \ldots, X_d]$ for $d>1$ is non-Noetherian, not integrally closed, and so forth (see above) indicates that these properties probably also fail for reciprocal complements of many other otherwise well-behaved integral domains. This provides a fertile ground for and source of problems for factorization theory and other investigations in non-Noetherian commutative algebra, being so different (as seen in the list of properties above) from valuation rings, pullbacks, integer-valued polynomial rings, and infinite-dimensional polynomial rings.

\section%{Properties of general rings of reciprocals}
{General properties}

In this section, we determine some properties of the \rco\ $R=R(T)$ of any integral domain $T$.  In particular, we show that $R$ is always local (see Theorem~\ref{thm:islocal!!}), that the  prime ideals of $R$ are generated by reciprocals of elements of $T$ (see Proposition~\ref{pr:gensprimes}), and that if $T$ is finitely generated over a field, then the fraction field of $R$ is a finitely generated $R$-algebra (see Proposition~\ref{pr:Gdom}).

\begin{defn}
For any integral domain $T$, we let $R(T)$ be the \emph{\rco} of $T$. That is, if $F(T)$ is the fraction field of $T$, then $R(T)$ is the subring of $F(T)$ generated by all terms of the form $1/f$, $f \in T \setminus \{0\}$. Equivalently, $R(T)$ is the set of all finite sums $\frac 1{f_1} + \cdots + \frac 1{f_t}$ where $t\geq 0$ and each $f_i \in T \setminus \{0\}$. 
\end{defn}

Let $T$ be an integral domain.  Let $E$ be the set of Egyptian elements of $T$, and set $G := E \cup \{0\}$.  Recall \cite[Proposition 8(3), where $G$ is called $E_3$]{GLO-Egypt} that $G$ is a subring of $T$.  Since it must then be an integral domain, $E = G \setminus \{0\}$ is a multiplicatively closed subset of $T$.

\begin{prop}
 \label{pr:reductiontoKalgebra}
 Let $T, E, G$ be as above.  Then:
 \begin{enumerate}
 \item[(1)] $R(E^{-1}T)= R(T)$, and
 \item[(2)] The set of Egyptian elements of $E^{-1}T$, along with $0$, coincides with the fraction field of $G$.
 \end{enumerate}
 \end{prop}
 
 \begin{proof}
 To prove (1), note that 
 taking the \rco\ preserves inclusion; hence, $R(E^{-1}T) \supseteq R(T)$. Conversely, let $y \in R(E^{-1}T)$. By definition there exists $d_1, \ldots, d_n \in T$ and $e_1, \ldots, e_n \in E$ such that 
 \[ y= \frac{e_1}{d_1} + \cdots + \dfrac{e_n}{d_n}. 
 \]
 But each $e_i$ is an Egyptian element of $T$, thus we can write $e_i = \sum_{j=1}^{n_i} \frac{1}{d_{i_j}}$. It follows that \[
 y = \sum_{i=1}^n \dfrac{1}{d_i} \left( \sum_{j=1}^{n_i} \frac{1}{d_{i_j}} \right) \in R(T).
 \]

 To prove (2), let $K$ the fraction field of $G$. 
% \bc we used $F$ as fraction field of $T$ earlier in definition 2.1, we can call $K$ the fraction field of $G$, as in the polynomial case \ec
 Pick an $(E^{-1}T)$-Egyptian element $x$ of $E^{-1}T$. Hence,  $x \in R(E^{-1}T)= R(T)$. We can write $x = \frac{d}{e}$ with $d \in T$ and $e \in E = G \setminus \{0\}$. But $e \in R(T)$ and therefore $d = ex \in R(T)$. It follows that $d \in R(T) \cap T = G$. Thus $x = \frac de \in K$.

 Conversely, let $0\neq x \in K$.  Then $x = \frac gh$ where $g,h \in E$.  Write $g = \sum_{i=1}^s \frac 1{d_i}$, $d_i \in T$.  Then $x=\frac gh = \sum_{i=1}^s \frac 1{d_i/g}$.  Since each $d_i/g \in E^{-1}T$, we have $x \in R(E^{-1}T)$.  Since $g \in T$ and $h\in E$, we have $x \in E^{-1}T$.  Hence $x$ is an Egyptian element of $E^{-1}T$.
 \end{proof}

 \begin{lemma}
\label{lem:units}
Let $K$ be a field.  Let $T$ be a $K$-algebra all of whose Egyptian elements are in $K$. Let $x_1, \ldots, x_n \in T \setminus K$ and  $u \in K \setminus \{ 0 \}.$ Then \[
 y= u + \dfrac{1}{x_1}+ \cdots + \dfrac{1}{x_n}
 \] is %either zero or
 a unit in $R:=R(T)$.
\end{lemma}

\begin{proof}
We can reduce to the case $u=1$ by dividing all the $x_i$'s by $u$.

If $n=0$, the statement is vacuously true. 
 So we may assume $n\geq 1$ and %If $n=1$, then $y\neq 0$ since otherwise $x_1 = -1 \in K$.  Hence, $x_1 + 1 \neq 0$, so we have \[ y^{-1}= \frac{x_1}{x_1 + 1} = 1 - \frac{1}{x_1+1} \in R(T).
%\]
%Then we can
work by induction on $n$.  Set $\alpha_i:= x_i^{-1}$ for each $1\leq i \leq n$.

If $y=0$, then $\alpha_n = -(1+ \sum_{i=1}^{n-1} \alpha_i) \in U(R)$ by inductive hypothesis.  Thus, $x_n=\alpha_n^{-1} \in R$, so $x_n$ is an Egyptian element of $T$, whence $x_n \in K$ by the assumptions on $T$.  But that contradicts the assumption on $x_n$ that $x_n \notin K$.  Hence, $y\neq 0$.

Next, notice that 
\[ H_n:= \frac{\prod_{i=1}^n \alpha_i}{1+ \alpha_1 + \ldots + \alpha_n } = \frac{1}{\prod_{i=1}^n x_i + \sum_{i=1}^n (\prod_{j \neq i}x_i)} \in R(T). \]
Starting from this fact we prove by reverse induction that $H_k := \frac{\prod_{i=1}^k \alpha_i}{1+ \alpha_1 + \ldots + \alpha_n } \in R(T)$ also for every $0 \leq k \leq n$. Suppose that $H_{k+1} \in R(T)$. Then notice that
\[
\prod_{i=1}^k \alpha_i - H_{k+1} = \frac{(\prod_{i=1}^k \alpha_i)(1+ \sum_{i\neq k+1} \alpha_i)}{1+ \alpha_1 + \ldots + \alpha_n} = (1+ \sum_{i\neq k+1} \alpha_i)H_k.
\]
Since by inductive hypothesis $1+ \sum_{i\neq k+1} \alpha_i$ is a unit in $R(T)$, we get $H_k \in R(T)$. In particular, $H_0 \in R(T)$, as was to be shown. %Finally, knowing that $H_1 \in R(D)$ and using a similar computation, we write 
%$$ 1 - H_1 = (1+ \alpha_2 + \ldots + \alpha_n) \frac{1}{1+ \alpha_1 + \ldots + \alpha_n}. $$ Using again the inductive hypothesis we obtain that $1+ \alpha_1 + \ldots + \alpha_n$ is a unit in $R(D)$.
\end{proof}

%%To relate this to the \rcos\ of polynomial rings over a field, we start with a proposition that is interesting in its own right.

\begin{thm}
 \label{thm:islocal!!}
 Let $T$ be an integral domain. Then $R(T)$ is a local ring, with maximal ideal generated by all elements of the form $\frac 1x$, with $0 \neq x\in T$ not an Egyptian element.
 \end{thm}

\begin{proof}
First assume that all the Egyptian elements of $T$ are in a subring $K$ of $T$ that is a field.  Let $\m$ be the set of finite sums of elements of the form $\frac 1x$, $x \in T \setminus K$.  Since any multiple of a nonunit of $T$ is a nonunit of $T$, and all units of $T$ are in $K$, it follows that $\m$ is an ideal of $R(T)$.  Moreover, any element of $R(T) \setminus \m$ is a unit of $R(T)$ by Lemma~\ref{lem:units}.  Thus it suffices to show that $\m$ does not contain a unit of $R(T)$.

Let $\alpha \in \m$.  Write $\alpha = \frac 1{x_1} + \cdots + \frac 1{x_n}$, with each $x_i \in T \setminus K$.  We proceed by induction on $n$ to show that $\alpha$ is not a unit.  If $n=0$ (so $\alpha =0$) the claim is vacuously true.  If $n>0$ and $\alpha$ is a unit of $R(T)$, then $\alpha^{-1} \in R(T)$.  We have \begin{equation}\label{eq:compoundinverse}
x_1 = (1/x_1)^{-1} = \frac 1{\alpha - \sum_{i=2}^n \frac 1{x_i}} = \frac{\alpha^{-1}}{1-\alpha^{-1}\sum_{i=2}^n \frac 1{x_i}}.
\end{equation}
But since $\alpha^{-1} \in R(T)$, $\sum_{i=2}^n \frac 1{x_i} \in \m$ by inductive hypothesis, and $\m$ is an ideal, we have $-\alpha^{-1}\sum_{i=2}^n \frac 1{x_i} \in \m$.  By Lemma~\ref{lem:units}, it follows that the denominator of (\ref{eq:compoundinverse}) is a unit.  Hence, $x_1 \in R(T)$, so that $\frac 1{x_1} \in U(R(T)) \subseteq K$, contradicting the fact that $x_1 \notin K$. 
 %. We proceed by induction on $n := \ell_T(\alpha)$ to show that $\alpha$ is not a unit, with the $n=0$ case (with $\alpha=0$) being vacuous.  So assume $n>0$ and the claim true for elements of $\m$ of shorter length.  Write $\alpha = \frac 1{x_1} + \cdots + \frac 1{x_n}$.

 Finally, we drop the assumption on $T$.  Let $E$ be the set of Egyptian elements of $T$.  Then by Proposition~\ref{pr:reductiontoKalgebra}, $R(T) = R(E^{-1}T)$, and $E^{-1}T$ has all its Egyptian elements in the subfield $E^{-1}G$, where $G$ is the subring $E \cup \{0\}$ of $T$.  Then by the first part of the proof, $R(T)$ is a local ring whose maximal ideal $\m$ is generated by all elements of the form $\frac 1{x/e}$, where $x \in T \setminus \{0\}$, $e \in E$, and $x/e \notin E^{-1}G$.  First note that since every $e\in E$ is a unit of $R(E)$, it follows that $\m$ is generated by those elements $\frac 1x$ where $x\in T \setminus \{0\}$ and $x \notin E^{-1}G$. But since $E^{-1}G \cap T = G$, the result follows. 
 %and since $E$ consists of units of $R(T)$, we have that $R(T)$ is a local ring with maximal ideal generated by all elements of the form $\frac 1x$, $0 \neq x \in T \setminus E$, as was to be shown.
\end{proof}

We will see shortly that every prime ideal of $R(T)$ shares the property with $\m$ that it is generated by reciprocals of elements of $T$.  First, we need the following notion of length.
\begin{defn}
    For $\alpha \in R(T)$, the $T$-\emph{length} of $\alpha$, denoted $\ell_T(\alpha)$ (or the \emph{length of} $\alpha$, $\ell(\alpha)$, if the ring is understood) is the minimum number $t$ such that there exist $f_1, \ldots, f_t \in T$ such that $\alpha = \frac 1{f_1} + \cdots + \frac 1{f_t}$.  For $\alpha$ in the fraction field of $T$ but not in $R(T)$, we write $\ell_T(\alpha) = \infty$.
\end{defn}

\begin{lemma}\label{lem:factorproduct}
Let $T$ be an integral domain and $0\neq \alpha \in R(T)$.  Write $\alpha = \sum_{i=1}^t \frac 1{f_i}$ with $t=\ell_T(\alpha)$ and each $f_i \in T \setminus \{0\}$.  Then in $R(T)$, $\alpha$ is a factor of the product of all the elements $1/f_i$.
\end{lemma}

\begin{proof}
Set $F := \prod_{i=1}^t f_i$.  Since $F/f_i \in T$ for each $i$, we have $F\alpha \in T$, and since both $F$ and $\alpha$ are nonzero elements of the fraction field of $T$, we have $F \alpha \neq 0$.  Hence $\frac 1{F\alpha} \in R(T)$.  Then the equation $\frac 1{F} = \frac 1{F\alpha} \cdot \alpha$ finishes the proof.
\end{proof}

\begin{lemma}
\label{lem:lengthminusone}
Let $T$ be an integral domain and $0 \neq \alpha \in R(T)$.  Let $t=\ell_T(\alpha)$ and write $\alpha = \sum_{i=1}^t \frac 1{f_i}$ with $f_i \in T \setminus \{0\}$.  Then $\ell_T(\alpha - \frac 1{f_t}) = t-1$.
\end{lemma}

\begin{proof}
Write $\beta = \alpha - \frac 1{f_t}$.  Since $\beta = \sum_{i=1}^{t-1} \frac 1{f_i}$, we have $\ell_T(\beta) \leq t-1$.  Write $\beta = \sum_{j=1}^s \frac 1{g_j}$, where $s=\ell_T(\beta)$ and each $g_j \in T \setminus \{0\}$.  Then $\alpha = \frac 1{f_t} + \sum_{j=1}^s \frac 1{g_j}$. Thus, \[
t=\ell_T(\alpha) \leq s+1 \leq (t-1)+1 = t.
\]
Hence, $s+1=t$, as was to be shown.
\end{proof}

As a consequence of the above two lemmas, we obtain the following result about the generators of any prime ideal of $R(T)$, which recapitulates the fact about the maximal ideal of $R(T)$ given in Theorem~\ref{thm:islocal!!}.

\begin{prop}
\label{pr:gensprimes}
Any prime ideal of $R(T)$ is generated by elements of the form $1/f$, $f \in T$.
\end{prop}

\begin{proof}
Let $0 \neq \alpha \in \p$, $t=\ell_T(\alpha)$.  We proceed by induction on $t$ to show that $\alpha$ is a sum of reciprocals of elements of $\p$.
%\bc I would write: to show that $\alpha$ is a sum of reciprocals that are in $\p$ \ec 

When $t=1$, it is clear. Suppose $t>1$.  Write $\alpha = \sum_{i=1}^t \frac 1{f_i}$, $f_i \in T$.  By Lemma~\ref{lem:factorproduct}, $\prod_{i=1}^t \frac 1{f_i}$ is a multiple of $\alpha \in R(D)$.  Hence, $\prod_{i=1}^t \frac 1{f_i} \in \p$.  Since $\p$ is prime, it follows that $\frac 1{f_i} \in \p$ for some $1\leq i \leq t$.  Let $\beta = \alpha - \frac 1{f_i}$. Clearly $\beta \in \p$, and by Lemma~\ref{lem:lengthminusone}, $\ell_T(\beta) = t-1$, so by inductive hypothesis, $\beta$ is a sum of elements of the form $\frac 1g \in \p$, with $g\in T$.  Thus, $\alpha = \beta + \frac 1{f_i}$ is also such a sum.
\end{proof}

We culminate this section with a result on \rcos\ of finitely generated $K$-algebras.

\begin{prop}\label{pr:Gdom}
Let $L/K$ be a field extension, let $f_1, \ldots, f_n \in L$, and let $T = K[f_1, \ldots, f_n]$.  Then $R(T)[\prod_{i=1}^n f_i] = \Frac T$.  Hence $1/\prod_{i=1}^n f_i \in \p$ for every nonzero prime ideal $\p$ of $R(T)$.
\end{prop}

\begin{proof}
Write $g = \prod_{i=1}^n f_i$. First, note that since $g \in T$, we have $1/g \in R(T)$.

We have $f_1 = g / (\prod_{i=2}^n f_i)  = g \cdot \frac 1{\prod_{i=2}^n f_i}\in R(T)[g]$, since $\frac 1{\prod_{i=2}^n f_i} \in R(T)$.  By symmetry, we have $f_1, \ldots, f_n \in R(T)[g]$.  Obviously $K \subseteq R(T)$ as well, so $T = K[f_1, \ldots, f_n] \subseteq R(T)[g]$. Let $\alpha \in \Frac T$.  We may write $\alpha = u/v$ with $u,v \in T$, $v\neq 0$.  Then since $u\in T \subseteq R(T)[g]$ and $1/v \in R(T) \subseteq R(T)[g]$, we have $\alpha = u\cdot \frac 1v \in R(T)[g]$.  Thus, $\Frac T \subseteq R(T)[g]$, but the reverse containment is obvious, so $R(T)[g]=\Frac T$.

The final statement follows from \cite[Theorem 19]{Kap-CR}.
\end{proof}

\begin{rem}\label{rem:Gdom}
Recall that a \emph{G-domain} is an integral domain whose fraction field is a finitely generated algebra over it \cite[Definition following Theorem 18]{Kap-CR}.  Hence, Proposition~\ref{pr:Gdom} implies that for any integral domain $T$ that is finitely generated over a field, $R(T)$ is a G-domain.  
\end{rem}

\section{Properties and bounds on the ring of polynomial reciprocals}
%\bc Section on general properties of rings of reciprocals, followed by polynomial ring of reciprocals properties section \ec
In this section, we give bounds on the \rco\ of a polynomial ring in $n$ variables.  That is, we exhibit rings that it is contained in and rings that it contains.  We also show it is atomic, but fails unique factorization.  A main tool is the map $\sigma$, an involution on $K(X_1, \ldots, X_n)$, which makes our ring isomorphic to an overring of $K[X_1, \ldots, X_n]$.

\begin{nota}\label{nota}
Let $D = D_n = K[X_1, \ldots, X_n]$, the polynomial ring in $n$ variables over a field $K$, where $n \geq 1$.  Let $F = F_n$ the fraction field of $D_n$. That is, $F_n = K(X_1, \ldots, X_n)$.
We set  $R := R_n = R(D_n)$.

%Let $D = K[X,Y]$, the polynomial ring in two variables over a field $K$.  Let $F = $ the fraction field of $D$. That is, $F = K(X,Y)$.

We let $\sigma = \sigma_n: F_n \ra F_n$ be the unique $K$-algebra homomorphism that sends $X_i \mapsto \frac 1{X_i}$ for $1\leq i \leq n$.  Note that $\sigma \circ \sigma = 1_F$; hence $\sigma$ is an \emph{involution}, whence a $K$-automorphism of $F$.  For any subring $T$ of $F$, we set $T^* := \sigma(T)$.

We define $2n$ functions $t_i, a_i: D \setminus \{0\} \ra \N_0$ for $1\leq i \leq n$ as follows.  We set $t_i(f) = c$ if $X_i^c | f$ but $X_i^{c+1} \nmid f$.  We let $a_i(f) = \deg_{X_i}(f) - t_i(f)$, where $\deg_{X_i}(f) =$ the degree of $f$ as a polynomial in $X_i$ with coefficients in $K[X_1, \ldots, X_{i-1}, X_{i+1}, \ldots, X_n]$.
%$a,b,c,d: D \setminus \{0\} \ra \N_0$ as follows.  If $X^c | f$ but $X^{c+1} \not|f$, we set $c(f) = c$.   If $Y^d | f$ but $Y^{d+1} \not |f$, we set $d(f) = d$.  Let $a(f) = \deg_X(f) - c(f)$, where $\deg_X(f) = $ the degree of $f$ as a polynomial in $X$ with coefficients in $K[Y]$.  Let $b(f) = \deg_Y(f)-d(f)$, where $\deg_Y(f) =$ the degree of $f$ as a polynomial in $Y$ with coefficients in $K[X]$.
Then for any $f \in D \setminus \{0\}$, we write $f = \left(\prod_{i=1}^n X_i^{t_i(f)}\right) f_0$, where $f_0 \in D \setminus \bigcup_{i=1}^n X_iD$.

For any $n$-tuple $(u_1, \ldots, u_n) \in \Z^n$, write $\mathbf u := (u_1, \ldots, u_n)$ and $\mathbf X^{\mathbf u} := \prod_{i=1}^n X_i^{u_i}$.
\end{nota}

Our first goal will be to prove that the maximal ideal is generated by the reciprocals of the nonconstant polynomials.

\begin{lemma}\label{lem:Egyptgraded}
Let $T$ be an $\N$-graded integral domain.  Then all the Egyptian elements of $T$ are in $T_0$, its 0th graded component.
\end{lemma}

\begin{proof}
For $f\in T$, we let its \emph{degree} be the degree of its largest nonzero graded component.  Note that degree is then additive in $T$; that is, if $f,g \in T \setminus \{0\}$, then $\deg(fg) = \deg(f) + \deg(g)$.  Also, if $f+g\neq 0$, then $\deg(f+g) \leq \max\{\deg f, \deg g\}$.

With this in mind, let $f\in T$ be Egyptian.  Write \[
f = \frac 1{f_1} + \cdots + \frac 1{f_s}
\]
with $f_j \in T$ for each $j$.  Clearing denominators by multiplying through by $\prod_{i=1}^s f_i$, we have \[
ff_1 \cdots f_s = \sum_{i=1}^s \prod_{j\neq i} f_j.
\]
By equating degrees, it follows that \[
\deg(f) + \sum_{i=1}^s \deg(f_i) \leq \max_i \sum_{j\neq i} \deg(f_j).
\]
Since all degrees are nonnegative, it follows that $\deg f = 0$, so that $f \in T_0$.
\end{proof}

\begin{thm}\label{thm:islocalpoly}
$R$ (as in Notation~\ref{nota}) is a local ring, with maximal ideal generated by all elements of the form $\frac 1f$, with $f$ a nonconstant polynomial.
\end{thm}

\begin{proof}
By Theorem~\ref{thm:islocal!!}, $R(D)$ is a local ring with maximal ideal $\m$ generated by all elements of the form $\frac 1f$ with $0 \neq f \in D$ non-Egyptian.  However, by Lemma~\ref{lem:Egyptgraded} (and using the standard grading on the polynomial ring), no nonconstant polynomial can be Egyptian.  Since the nonzero constant polynomials are units and hence Egyptian, the result follows.
%
%This follows from Theorem~\ref{thm:islocal!!} and Lemmas~\ref{lem:Egyptgraded} and \ref{lem:unitspoly}.
\end{proof}

\begin{prop}\label{pr:downdim}
Let $0 \leq j < n$ be integers.  Then $R_n \cap F_j = R_j$ and $R_n^* \cap F_j = R_j^*$.
\end{prop}

\begin{proof}
We need only prove the first statement, since it then follows that $R_n^* \cap F_j = \sigma(R_n) \cap F_j = \sigma(R_n \cap F_j) = \sigma(R_j) = R_j^*$.  Moreover, by an easy induction, we may assume $j=n-1$.

It is clear that $R_{n-1} \subseteq R_n \cap F_{n-1}$.  So let $\alpha \in R_n \cap F_{n-1}$.  Then $\alpha = \sum_{i=1}^t \frac 1 {f_i}$, where each $f_i \in D_n \setminus \{0\}$.  Reorder the $f_i$ such that $f_1, \ldots, f_s \in D_{n-1}$ and $f_{s+1}, \ldots, f_t \in D_n \setminus D_{n-1}$. Then for $1\leq i \leq s$, we have $\frac 1{f_i} \in R_{n-1} \subseteq R_n \cap F_{n-1}$.  Let $\beta = \alpha - \sum_{i=1}^s \frac 1 {f_i}$; then we have \begin{equation}\label{eq:beta}
\beta = \sum_{i=s+1}^t \frac 1{f_i}.
\end{equation}

Assume $\beta \neq 0$.  Then multiplying Equation~\ref{eq:beta} by $\prod_{i=s+1}^t f_i$, we have 
\[
\beta f_{s+1} \cdots f_t = \sum_{i=s+1}^t \prod_{\substack{j\neq i \\ j>s}} f_j.
\]
With respect to the polynomial ring $F_{n-1}[X_n]$, note that $\beta \in F_{n-1}$ and each $f_i$ for $i>s$ is a nonconstant polynomial.  Say $\deg f_i = d_i >0$ for each $i>s$.  Then the left hand side above has degree $\sum_{i=s+1}^t d_i$, whereas the right hand side has degree $\leq \max_i \{\sum_{j\neq i,\ j>s} d_j\} < \sum_{s=i+1}^t d_i$, a contradiction.  Hence $\beta=0$.  That is, $\alpha = \sum_{i=1}^t \frac 1{f_i} \in R_{n-1}$, since each $f_i \in D_{n-1}$.
%
%
%so that by subtracting off this sum from $\alpha$, we may assume $s=0$ and $f_1, \ldots, f_t \in D_n \setminus D_{n-1}$.
%
%Assume $t>0$.  Then multiplying the equation $\alpha = \sum_{i=1}^t \frac 1{f_i}$ by $\prod_{i=1}^t f_i$, we have \[
%\alpha f_1 \cdots f_t = \sum_{i=1}^t \prod_{j\neq i} f_j.
%\]
%With respect to the polynomial ring $F_{n-1}[X_n]$, note that $\alpha \in F_{n-1}$ and each $f_i$ is a nonconstant polynomial.  Say $\deg f_i = d_i >0$.  Then the left hand side above has degree $\sum_{i=1}^t d_i$, whereas the right hand side has degree $\leq \max_i \{\sum_{j\neq i} d_j\} < \sum_{i=1}^t d_i$, a contradiction.  Hence $t=0$, so $\alpha=0$... [[RESTATE CLEARER WITHOUT WMA?]]
\end{proof}

\begin{lemma}\label{lem:star}
Let $f \in D \setminus \{0\}$.  Then \[
\sigma\left(\frac 1f\right) = \frac{\mathbf X^{\mathbf a(f) + \mathbf t(f)}}{f^*},
\]
where $f^* \in D \setminus \bigcup_{i=1}^n X_i D$.  Moreover, $\mathbf a(f) = \mathbf a(f^*)$ and $f = \mathbf X^{\mathbf t(f)} f^{**}$.
\end{lemma}

\begin{proof}
First suppose $f \in D \setminus \bigcup_{i=1}^n X_iD$, so that $\mathbf t(f) = \mathbf 0$.  Write $f = \sum_{\mathbf j \in \N_0^n} u_{\mathbf j} \mathbf X^{\mathbf j}$, where each $u_{\mathbf j} \in K$ and $u_{\mathbf j} = 0$ for all but finitely many $n$-tuples $\mathbf j$.  Then $a_i(f) = \deg_{X_i}(f) = \max \{c \mid \exists \mathbf j$ with $j_i = c$ and $u_{\mathbf j} \neq 0\}$ for each $1\leq i \leq n$.  We have \[
\sigma\left(\frac 1f\right) = \frac 1 {\sum_{\mathbf j} u_{\mathbf j} / \mathbf X^{\mathbf j}} = \frac{\mathbf X^{\mathbf a(f)}}{\sum_{\mathbf j} u_{\mathbf j} X^{\mathbf a(f) - \mathbf j}}.
\]
Let $f^*$ denote the expression in the denominator above.  Note that $\mathbf a(f) - \mathbf j \in \N_0^n$ whenever $u_{\mathbf j} \neq 0$, since for each such $\mathbf j$ we have $j_i \leq a_i(f)$ for all $1\leq i \leq n$.  Hence $f^*$ is a true polynomial.  Moreover, for each $i$, since $X_i \nmid f$, there is some $\mathbf j$ with $u_{\mathbf j} \neq 0$ and $j_i=0$, and hence $a_i(f) - j_i = a_i(f)$.  Thus, $a_i(f^*) = a_i(f)$. Finally, for each $i$, there is some $\mathbf j$ with $u_{\mathbf j} \neq 0$ and $j_i = a_i(f)$. Hence $a_i(f) - j_i = 0$, so $X_i \nmid f^*$, whence $f_i^* \in D \setminus \bigcup_{i=1}^n X_i D$.

For the final claim, we have \[
\sigma\left( \frac 1 {f^*}\right) = \frac {\mathbf X^{\mathbf a(f^*)}}{\sum_{\mathbf j} u_{\mathbf j} \mathbf X^{\mathbf a(f^*) - (\mathbf a(f) - \mathbf j)}} = \frac{\mathbf X^{\mathbf a(f^*)}}{\sum_{\mathbf j} u_{\mathbf j} \mathbf X^{\mathbf j}} = \frac{\mathbf X^{\mathbf a(f^*)}}f.
\]

Now we go to the general case, where $\mathbf t(f)$ is not necessarily the zero vector.  We have $f = \mathbf X^{\mathbf t(f)} f_0$, where $f_0 \in D \setminus \bigcup_{i=1}^n X_iD$.  Then \[
\sigma\left(\frac 1f\right) = \left(\prod_{i=1}^n \sigma\left(\frac 1{X_i}\right)^{t_i(f)}\right) \sigma\left(\frac 1{f_0}\right) = \mathbf X^{\mathbf t(f)} \sigma\left(\frac 1{f_0}\right) = \frac{\mathbf X^{\mathbf a(f_0) + \mathbf t(f)}}{f_0^*}.
\]
Moreover, $a_i(f) = \deg_{X_i}(f) - t_i(f) = \deg_{X_i}(f_0) = a_i(f_0)$ for each $1\leq i \leq n$, so $\mathbf a(f) = \mathbf a(f_0)$.  Setting $f^* = f_0^*$, we have $f^{**} = f_0^{**} = f_0$, so that $f = \mathbf X^{\mathbf t(f)} f_0 = \mathbf X^{\mathbf t(f)} f_0^{**}$, completing the proof.
\end{proof}

\begin{lemma}
\label{Xiprimes}
We have $R^* \subseteq D_{(X_i)}$ for each $1\leq i \leq n$. In particular there are $n$ distinct height one prime ideals $\p_i$ of $R^*$ obtained as centers of the $X_i$-adic valuations of $D$, and $R^*_{\p_i} = D_{(X_i)}$. %\bc I think here we have $R^*_{\p_i} = D_{(X_i)}$, not just isomorphic since we are working with $R^*$ \ec
\end{lemma}

\begin{proof}
Choose $i$ with $1\leq i \leq n$.  Let $f\in D$. %Write $f = \mathbf X^{\mathbf t(f)} f_0$, where no $X_j$ divides $f_0$. 
Let $v_i$ be the $X_i$-adic valuation function.  Then by Lemma~\ref{lem:star}, $v_i(\sigma(1/f)) = t_i(f) + a_i(f) - v_i(f^*) \geq 0$, as $a_i(f) = \deg_{X_i}(f^*)$, and $v_i(f^*)$ cannot exceed the $X_i$-degree of $f^*$. Since every nonzero element $\alpha \in R^*$ is a sum of terms of the form $\sigma(1/f)$, it follows that $v_i(R^*) \geq 0$.

Now let $1 \leq i < j \leq n$ and let $\p_i$, $\p_j$ be the centers of the $X_i$ and $X_j$-adic valuations on $D$ in $R^*$, respectively.  Since $v_i(X_i)=1$ but $v_j(X_i)=0$, we have $X_i \in \p_i \setminus \p_j$. Similarly, $v_i(X_j) = 0$ and $v_j(X_j)=1$, so $X_j \in \p_j \setminus \p_i$.

For the final claim, let $\alpha \in D_{(X_i)}$.  Then $\alpha=f/g$ for some $f\in D$ and $g \in D \setminus X_iD$.  If $g \in \p_i$, then $g \in X_iD_{(X_i)} \cap D = X_i D$, which is a contradiction.  Hence, $f \in R^*$ and $g \in R^* \setminus \p_i$, so $f/g \in R^*_{\p_i}$.  Thus, $R^*_{\p_i} = D_{(X_i)}$, whence $\hgt \p_i =1$.
\end{proof}

\begin{lemma}\label{lem:Xinallprimes}
Let $\p$ be a nonzero prime ideal of $R$.    Then for some $1\leq i \leq n$, $1/X_i \in \p$.
\end{lemma}

\begin{proof}
By Proposition~\ref{pr:Gdom}, $1/(\prod_{i=1}^n X_i) = \prod_{i=1}^n (1/X_i) \in \p$.  Since $\p$ is prime, some $1/X_i \in \p$.
\end{proof}

\begin{lemma}
\label{nonzeroconstantterm}
Let $f \in D$. If $f(\mathbf 0) \neq 0$ then $f$ is a unit in $R^*$. In particular \[
K[X_1, \ldots, X_n]_{(X_1, \ldots, X_n)} \subseteq R^*.\]
\end{lemma}

\begin{proof}
We prove this by induction on $n$. If $n=0$, then the result is vacuous.  Thus, let $n\geq 1$ and assume the result true for smaller $n$.%If $n=1$, then $R^* = K[X]_{(X)}$ by \cite{nme-Euclidean}. Thus, assume $n\geq 2$ and assume the result true for smaller $n$.

By way of contradiction suppose that $f \in \p$ for a prime ideal $\p$ of $R^*$.
By Lemma~\ref{lem:Xinallprimes}, some $X_i \in \p$; without loss of generality assume $i=n$, so that $X_n \in \p$.  Then $f = X_n g + h$ for some $g \in D_n$ and $h \in D_{n-1} \setminus (X_1, \ldots, X_{n-1})D_{n-1}$.  Then $h \in \p \cap D_{n-1} \setminus (X_1, \ldots, X_{n-1})D_{n-1}$, so that $1/h \in K[X_1, \ldots, X_{n-1}]_{(X_1, \ldots, X_{n-1})} \subseteq R_{n-1}^*$ by inductive hypothesis.  Hence also $1/h \in R^*$.  But then $1 = (1/h)(f-X_n g) \in \p$, a contradiction.
\end{proof}

\begin{lemma}\label{lem:ordervaluation}
Let $w$ be the \emph{order valuation} of $D_{(X_1, \ldots, X_n)}$ on $F$ -- i.e., the unique valuation on $F$ such that for any nonzero $g \in D$, $w(g) = \max\{j \mid g \in (X_1, \ldots, X_n)^j\}$.  Let $(W, \m_W)$ be the corresponding DVR.  
Then $R^* \subseteq W$, and $\m_W \cap R^*$ is the maximal ideal $\m$ of $R^*$.
\end{lemma}

\begin{proof}
By construction $D \subseteq W$ and $X_i \in \m_W$ for all $i$, so that $w(X_i)>0$.  We first show that for any nonconstant polynomial $f$ not divisible by any of the variables,  $w(\sigma(1/f)) \geq 1$. Under these assumptions,  $\sigma(1/f) = \frac{\mathbf X^{\mathbf a(f)}}{f^*}$ by Lemma~\ref{lem:star}.  Note that there is some $i$ and some monomial $m$ in $f^*$ such that $X_i^{a_i(f)} \nmid m$.  Thus, $w(f^*) \leq w(m) \leq (a_i(f) - 1) + \sum_{j\neq i} a_j(f) < \sum_j a_j(f)$.  Thus, $w(\sigma(1/f)) = \sum_j a_j(f) - w(f^*) >0$.

For a general nonconstant polynomial $f$, we have $f = \mathbf X^{\mathbf t(f)}f_0$ where $f_0$ is not a multiple of any of the $X_i$, and if $f_0$ is constant then some $t_i(f)>0$.  Hence, $w(\sigma(1/f)) = \sum_{i=1}^n t_i(f) w(X_i) + w(\sigma(1/f_0))>0$.

Now let $\alpha \in R^*$.  Write $\alpha = u + \sum_{j=1}^t \sigma\left(\frac 1{f_j}\right)$, where $u \in K$ and each $f_j \in D \setminus K$.  Since $K \subseteq D \subseteq W$, by the above we have $\alpha \in W$, whence $R^* \subseteq W$.  If $u=0$ then $w(\alpha) \geq \min\{w(\sigma(1/f_i) \mid 1\leq i \leq t\} >0$, so that $\alpha \in \m_W$.  Thus by Theorem~\ref{thm:islocalpoly}, $\m \subseteq \m_W \cap R^*$, but then since $\m$ is maximal, the result follows.
\end{proof}

%\begin{thm}\label{thm:islocal}
%The ring $R^*$ is local, with unique maximal ideal $\m$ generated by the elements $\sigma(1/f)$ for $f\in D\setminus K$, and residue field isomorphic to $K$.
%\end{thm}

%\begin{proof}
%Let $\m$ be the ideal generated by the $\sigma(1/f)$.  Let $(W,\m_W)$ be as in Proposition~\ref{pr:ordervaluation}.  Then by that result, $\m \subseteq \m_W \cap R^*$, whence $\m$ is a proper ideal of $R^*$.

%On the other hand, take any $\alpha \in R^*$.  We have $\alpha = c + \sum_{i=1}^m \sigma(1/f_i)$, where $c\in K$, $m\geq 0$, and each $f_i \in D \setminus K$.  If $c=0$, then $\alpha \in \m$.  But if $c\neq 0$, then by Lemma~\ref{lem:unit2}, $\alpha$ is a unit of $R^*$. Thus, the ideal $\m$ consists precisely of the nonunits of $R^*$, and is therefore the unique maximal ideal of $R^*$.

%For the final claim, let $\beta \in R^*$.  Then $\beta = u + \sum_{j=1}^t \sigma(1/f_j)$ for some $u\in K$ and $f_j \in D \setminus K$.  Since each $\sigma(1/f_j) \in \m$, we have $u \mapsto \bar\beta$ in the composition $K \hookrightarrow R^* \onto R^* /\m$.  Thus, the composition in question is surjective, and since it is a homomorphism of fields, it must also be injective, hence an isomorphism.
%\end{proof}
The following result must be well-known but we provide a proof for the convenience of the reader.

\begin{lemma}\label{lem:atomiccenter}
Let $(T,\m)$ be a local integral domain with fraction field $F$, and let $(V,\n)$ be a discrete rank one valuation ring such that $T \subseteq V \subseteq F$ and $\n \cap T = \m$.  Then $T$ is atomic, and any $x \in \m \setminus \n^2$ is an irreducible element of $T$.
\end{lemma}

\begin{proof}
We begin with the second statement.  Let $x\in \m \setminus \n^2$.  Write $x=st$ with $s, t \in T$ and $s$ not a unit.  Then if $v$ is the valuation function of $v$, we have $v(s) \geq 1$, so that $1 = v(x) = v(st) = v(s) + v(t) \geq 1+v(t)$, whence $v(t) = 0$, so that $t \in T \setminus \n = T \setminus \m$ and is thus a unit.  Thus, $x$ is irreducible.

For the first statement, let $A = \{x \in \m \mid x $ cannot be written as a product of irreducible elements$\}.$  If $A \neq \emptyset$, choose $a\in A$ such that $v(a) \leq v(b)$ for all $b\in A$.  Since $a$ is not irreducible, we may write $a=bc$ for some nonunits $b,c \in T$; hence $b,c \in \m$.  But then $v(c) \geq 1$, so $v(b) = v(a)-v(c) < v(a)$.  Thus, $b \notin A$, so $b$ can be written as a product of irreducible elements.  By the same argument, the same holds for $c$.  Hence, $a=bc$ is also a product of irreducible elements, which is a contradiction.  Thus, $A = \emptyset$, so every element of $T$ is either a unit or a product of irreducibles.  That is, $T$ is atomic.
\end{proof}

\begin{thm}\label{thm:atomic}
The ring $R$ is atomic.  That is, every nonzero nonunit element factors into a product of irreducible elements.
\end{thm}

\begin{proof}
Since $R \cong R^*$, we may work with $R^*$. By Lemma~\ref{lem:ordervaluation}, the maximal ideal of $R^*$ is the center of a rank 1 discrete valuation. The result then follows from Lemma~\ref{lem:atomiccenter}.
\end{proof}

\begin{rem}\label{rem:notUFD}
However, $R$ is not a UFD provided $n\geq 2$.  To see this (and working in $R^*$), use the labels $X := X_1$ and $Y := X_2$, and first note that $s := \sigma(1/(X+Y)) = XY / (X+Y) \in R^*$, and also that $t:= X^2/(X+Y) = X-s$, $u := Y^2/(X+Y) = Y-s \in R^*$.  But each of $s,t,u$ has value 1 in the order valuation $w$ from Lemma~\ref{lem:ordervaluation}, hence must be irreducible elements of $R^*$ by Lemma~\ref{lem:atomiccenter}.

Then we have $\frac{X^2 Y^2}{(X+Y)^2} = s^2 = tu$, so if $R^*$ were a UFD, $s$ would be associate to either $t$ or $u$.  But if $s$ is an associate of $t$, then $s \in tR^*$, which implies that $\frac YX \in R^*$.  And if $s$ is an associate of $u$, then $u \in sR^*$, which again implies that $\frac YX \in R^*$.  So in either case we obtain $X/Y = \sigma(Y/X) \in R = R(D) \subseteq R(K(X_3, \ldots, X_n)[X, Y])$, contradicting \cite[Example 2.9]{nme-Euclidean}.
\end{rem}

\section{Dimension}

In this section, we show that $R$ has the same Krull dimension as $D$ (see Theorem~\ref{thm:dim}).

\begin{lemma}\label{lem:localizingatonevariable}
For any $1 \leq i \leq n$, we have \[
R[X_i] = R(K(X_i)[X_1, \ldots, \widehat{X_i}, \ldots, X_n]).
\]
Hence, $R^*[X_i^{-1}]$ is isomorphic to the \rco\ of the polynomial ring in $n-1$ variables over a field.
\end{lemma}

\begin{proof}
Without loss of generality set $i=n$.  Set $S := R(K(X_n)[X_1, \ldots, X_{n-1}])$.

For the forward containment, first note that $X_n = \frac 1{1/X_n} \in S$ since $1/X_n \in K(X_n)[X_1, \ldots, X_{n-1}]$.  Moreover, since $D_n \subseteq K(X_n)[X_1, \ldots, X_{n-1}]$ and $R(-)$ preserves containment, we have $R \subseteq S$.  Thus, $R[X_n] \subseteq S$.

For the reverse, let $0 \neq f \in K(X_n)[X_1, \ldots, X_{n-1}]$.  By finding a common denominator to the $K(X_n)$-coefficients of the monomials in $X_1, \ldots, X_{n-1}$, we may write $f=g/h$, where $g \in D_n$ and $h \in K[X_n]$.  Write $h = \sum_{i=0}^t c_i X_n^i$ with all $c_i \in K$.  Then \[
\frac 1f = \frac hg = \sum_{i=0}^t \frac{c_i}g \cdot X_n^i.
\]
But for each $i$, if $c_i=0$, then $\frac {c_i}g=0 \in R_n$; otherwise $\frac {c_i}g = \frac 1{g/c_i} \in R$. Thus, $\frac 1f = \sum_{i=0}^t \frac{c_i}g \cdot X_n^i \in R[X_n]$.  Hence, $S \subseteq R[X_n]$, completing the proof.
\end{proof}

\begin{prop}\label{pr:loclower}
Let $j,n$ be integers with $0\leq j \leq  n$.  Then there is a unique prime ideal $P$ of $R$ such that $R_P$ is the \rco\ of $L[X_1, \ldots, X_j]$, where $L = K(X_{j+1}, \ldots, X_n)$.
\end{prop}

\begin{proof}
By applying induction to Lemma~\ref{lem:localizingatonevariable}, we have $R[X_{j+1}, \ldots, X_n] = R(L[X_1, \ldots, X_j]) =: S$.  But since $S$ is a local ring by Theorem~\ref{thm:islocal!!}, and since each $X_i^{-1} \in R$, there is a unique prime ideal $P$ of $R$ maximal with respect to avoiding all of $X_i^{-1}$ for $j+1\leq i \leq n$ such that $R_P = S$.
\end{proof}

\begin{lemma}\label{lem:Xiprimes}
For each $i$, there is a unique prime ideal $\q_i$ of $R$ maximal with respect to not containing $1/X_i$. We have $R_{\q_i} = R[X_i]$.
\end{lemma}

\begin{proof}
We may assume $i=n \geq 1$.  By Lemma~\ref{lem:localizingatonevariable}, $R[X_n]$ is the \rco\ of the polynomial ring in $n-1$ variables over a field, which by Theorem~\ref{thm:islocal!!} is local.  Hence by elementary localization theory, there is a unique prime ideal $\q_n$ of $R$ maximal with respect to avoiding $X_n^{-1}$, and $R_{\q_n} = R[X_n] = R(K(X_n)[X_1, \ldots, X_{n-1}])$. %Then apply $\sigma$ to both sides of the equation, and set $\q_n := \sigma(\p)$.
\end{proof}

\begin{thm}\label{thm:dim}
We have $\dim R = n$.
\end{thm}

\begin{proof}
We proceed by induction on $n$.  Of course $R_0 = K$, which has dimension 0, so we may assume $n>0$. By Theorem~\ref{thm:islocalpoly}, $R$ is local, and its maximal ideal $\m$ contains the reciprocals of all the variables.  Let $\q=\q_n$ be as in Lemma~\ref{lem:Xiprimes}, so that $1/X_n \notin \q_n$ and $R_\q$ is isomorphic to the \rco\ of a polynomial ring in $n-1$ variables over a field by Lemma~\ref{lem:localizingatonevariable}.  Then by inductive hypothesis, $\hgt \q = \dim R_\q = n-1$.  Since $1/X_n \in \m \setminus \q$, we have $\hgt \m > \hgt \q=n-1$, whence $\dim R=\hgt \m \geq n$.

On the other hand, since $R^*$ is an overring of the $n$-dimensional Noetherian domain $D$, we have $\dim R = \dim R^* \leq n$; see \cite{ABDFK-Jaffard}.  Hence $\dim R = n$.
\end{proof}

%\begin{cor}\label{cor:nonNoeth}
%If $n\geq 2$, then $R_n^*$ is non-Noetherian.
%\end{cor}

%\begin{proof}
%Suppose $R_n^*$ were Noetherian.  Then since $R_n^*$ is a Noetherian domain of Krull dimension $\geq 2$ by Theorem~\ref{thm:dim}, $R_n^*$ has infinitely many height one primes by \cite[Theorem 144]{Kap-CR}. Then by Proposition~\ref{pr:Xinallprimes} and the pigeonhole principle, there must be some $1\leq i \leq n$ such that $X_i$ lies in infinitely many height one primes, an impossibility for a nonzero element of  a Noetherian domain.
%\end{proof}

\section%[The 2-dim case]{Special properties of the 2-dimensional case and their implications}
%{$R^*$ is not integrally closed}
{Exotic properties of $R$}
%\bc should the title of this section be this one? \ec

For most of this section, we work in 2 variables, so that $D = K[X,Y]$ where  $X=X_1$, $Y=X_2$ for short.  Then the notation $D_n$, $R_n$, etc. when $n\neq 2$ will stand for the corresponding rings of other dimensions.  We show that $R_n^*$ is not integrally closed when $n\geq 2$. We also show that it is not a finite conductor domain, hence not coherent, and is thus also non-Noetherian.

Before we begin, recall the following presumably well-known result: \begin{lemma}\label{lem:free}
Let $A \subseteq B$ be integral domains such that $B$ is free as an $A$-module.  Then $B \cap \Frac(A) = A$.
\end{lemma}

\begin{proof}
Let $b \in B \cap \Frac{A}$. Write $b = x/y$ with $x,y\in A$ and $y\neq 0$.  Then $x=yb \in yB \cap A = (yA)B \cap A = yA$, where the latter equation holds by freeness.  Thus, $x=ya$ for some $a\in A$, whence $yb=ya$, so by cancellation, $b=a \in A$.
\end{proof}

Our methodology here is to construct a family of valuation rings that contain $R^*$, which serve as a tool to analyze the elements and prime ideals of our ring.  We must start with notation that will be useful:

%\bc here I would mention that we need some technical work to define a family of valuation overrings of $R^*$ \ec

\begin{nota}\label{nota:relpr}
Choose two relatively prime positive integers $p,q$ with $p<q$, such that neither $p$ nor $q$ is a multiple of $\chr K$. %  (For instance, one can choose two of the three numbers $7,11,13$.)
Let $K'$ be the smallest field extension of $K$ that contains all the primitive $p$th and $q$th roots of 1.  Let $L/K'(X,Y)$ be generated by elements $s,t$ such that $s^p=X$ and $t^q=Y$.  Note that $K'[s,t]$ is free as a $K'[X,Y]$-module on the basis $\{s^i t^j \mid 0 \leq i <p,$ $0\leq j<q\}$.
\end{nota}

\begin{lemma}\label{lem:udiv}
Let $g \in K[X,Y]$.  Then $s-t \mid g$ in $K'[s,t]$ if and only if $X^q - Y^p \mid  g$ in $K[X,Y]$.
\end{lemma}

\begin{proof}
Since $K'[X,Y] \cap K(X,Y) =K[X,Y]$ by Lemma~\ref{lem:free}, we may assume $K=K'$.  Thus, we may let $\xi_p$ (resp. $\xi_q$) be a primitive $p$th (resp. $q$th) root of unity in $K$.  Then \begin{align*}
    \prod_{i=0}^{p-1} \prod_{i=0}^{q-1} (\xi_p^i s - \xi_q^j t) &= \prod_{i=0}^{p-1} \left( (\xi_p^is)^q - t^q\right) = \prod_{k=0}^{p-1} (\xi_p^k s^q - t^q)
    = (-1)^p\prod_{k=0}^{p-1} (t^q - \xi_p^k s^q) \\&= (-1)^p(t^{pq} - s^{pq}) = (-1)^p(Y^p - X^q).
\end{align*}
%[[DOES THIS NEED MORE EXPLANATION?]]
The second equality above holds because $p,q$ are relatively prime, so that the order of $q+p\Z$ in $\Z/p\Z$ must be $p$.

Also note that for each pair $(i,j)$ of integers with $0 \leq i <p$, $0 \leq j <q$, there is unique $\tau_{ij} \in \Aut_{K(X,Y)}L$ such that $\tau_{ij}(s)= \xi_p^is$ and $\tau_{ij}(t)=\xi_q^jt$.
%in the action of the Galois group of $L$ over $F$, the conjugates of $s-t$ are precisely the elements $\xi_p^i s - \xi_q^jt$ for $0\leq i<p$ and $0 \leq j<q$.
Thus, if $s-t \mid g$ in $K[s,t]$, then for each $i,j$, we have %$\tau \in \Aut_F(L)$ such that 
$\tau_{ij}(s-t) = \xi_p^i s - \xi_q^jt \mid \tau_{ij}(g) = g$.  Since the $\xi_p^i s - \xi_q^jt$ are mutually nonassociate irreducible elements of $K[s,t]$, a UFD, it follows that $X^q - Y^p = \pm\prod_{i,j} (\xi_p^i s - \xi_q^jt) \mid g$ in $K[s,t]$.  Hence, \[
\frac g{X^q - Y^p} \in K(X,Y) \cap K[s,t] = K[X,Y],
\]
again by Lemma~\ref{lem:free},
which implies that $X^q - Y^p \mid g$ in $K[X,Y]$.

For the converse, simply note that if $X^q - Y^p \mid g$ in $K[X,Y]$, then as $s-t \mid X^q - Y^p$ in $K[s,t]$ and $K[X,Y] \subset K[s,t]$, it follows by transitivity of divisibility that $s-t \mid g$ in $K[s,t]$.
\end{proof}

\begin{nota}\label{nota:val} Let $u:=s-t$.  Then $s,u$ are algebraically independent over $K'$, and $K'[s,t] = K'[s,u]$.  We define a valuation $w:= w_h := w_{p,q,h}$ on $K'[s,u]$ by setting $w(s) = 1$, $w(u) = h$ for some integer $h \geq1$, and for any nonzero $f = \sum_{i,j} c_{ij} s^i u^j$ in $K'[s,t]$, where $c_{ij} \in K'$, we set $w(f) = \min \{w(s^i u^j) \mid c_{ij} \neq 0\} = \min \{i+hj \mid c_{ij} \neq 0 \}$.  Then we let $W := W_h := W_{p,q,h}$ be the corresponding valuation ring in the field $L$.  Clearly $K'[s,u] \subseteq W$.  Set $V := W \cap F$ (denoted $V_h$ or $V_{p,q,h}$ if needed) and let $v =v_h = v_{p,q,h}$ be the corresponding valuation on $F$.
\end{nota}

%\begin{prop}\label{pr:overp}
%Let $f$ be a nonconstant irreducible polynomial in $K[X,Y]$ and set $\alpha = \sigma(1/f^*)$.  If $f$ is a constant multiple of $X^q-Y^p$ then $v(\alpha) = 0$; otherwise $v(\alpha)\geq p$.  Hence, $V$ is a valuation overring of $R^*$.
%\end{prop}
%[[ABOVE IS WRONG]]

\begin{lemma}
\label{lem:overp}
The valuation ring $V$ is an overring of $R^* := R_2^*$ if and only if $h \leq pq+1$.  If $h<pq+1$, then $\sigma(1/f) \in \m_V$ for all $f \in D \setminus K$.

Suppose on the other hand that $h=pq+1$. 
Then for an irreducible polynomial $f \in K[X,Y]$, with $\alpha= \sigma(\frac{1}{f^*})$ and $ \theta = \sigma(\frac{1}{X^q-Y^p}) = \frac{X^qY^p}{Y^p-X^q}$, we have $v(\alpha)=0$ if and only if
$\alpha = \theta^m \delta^{-1}$ for some $m \geq 1$ and some element $\delta \in K[\theta] \setminus (\theta)K[\theta]$. %, then
%$f= Y^2-X^3$, then $v(\alpha)=0$, 
Otherwise $v(\alpha) \geq p$. 
%In particular $V$ is a valuation overring of $R^*$.
\end{lemma}

%[[NOTE: I'm not convinced of the ``otherwise'' statement.]]

\begin{proof}
Since $X=X^*$ and $Y=Y^*$, we have $v(\sigma(1/X^*)) = v(X) = w(s^p) = pw(s) = p$, and $w(t) = w(s-u) = 1$, so $v(\sigma(1/Y^*)) = v(Y) = w(t^q) = qw(t) = q$.  Hence, $V$ is an overring of $K[X,Y]$, with $(X,Y) \subseteq \m_V$.

Now suppose $(i,j)$ is a pair of integers with $0 \leq i < p$, $0 \leq j < q$, and $(i,j) \neq (0,0)$.  Then $\xi_p^is - \xi_q^jt = (\xi_p^i - \xi_q^j)s + \xi_p^j u$, so since $\xi_p^i - \xi_q^j \in K \setminus  \{0\}$, we have $w(\xi_p^is - \xi_q^jt)=1$.  Therefore, \begin{align*}
v(X^q-Y^p) &= w\left(\pm \prod_{i=0}^{p-1} \prod_{i=0}^{q-1} (\xi_p^i s - \xi_q^j t) \right) = w(s-t) + \sum_{i,j}^{(i,j) \neq (0,0)} w(\xi_p^i s - \xi_q^j t) \\
&= h + pq-1.
\end{align*}

It follows that \[ 
v(\theta) = 
v\left(\frac{X^qY^p}{X^q-Y^p}\right) = qp + pq - (h+pq-1) = pq+1-h.
\]
Thus, if $h>pq+1$, we have $v(\theta)<0$ so that $\theta \notin V$ and $R^* \nsubseteq V$.  But as long as $h\leq pq+1$ we have $v(\theta) \geq 0$, with $v(\theta) >0 \iff h<pq+1$.  From now on we assume $h \leq pq+1$.

Now let $f \in K[X,Y]$ be nonconstant, irreducible, and not associate to any of $X$, $Y$, $X^q-Y^p$. Then for some  $c_{ij} \in K$, we have \begin{align}\label{al:suXY}
\begin{split}
    f &= \sum_{i,j} c_{ij} X^i Y^j = \sum_{ij} c_{ij} s^{pi} (s-u)^{qj} \\
    &= \sum_{i,j} c_{ij} s^{pi} \sum_{k=0}^{qj} (-1)^k {qj \choose k} s^{qj-k} u^k = \\
    &= \left(\sum_{i,j} c_{ij} s^{pi+qj}\right) + u \cdot \sum_{i,j} c_{ij} s^{pi} \sum_{k=1}^{qj} (-1)^k {qj \choose k} s^{qj-k} u^{k-1}.
\end{split}
\end{align}
Then since $f$ is not associate to (hence not divisible by) $X^q - Y^p$ in $K[X,Y]$, it follows from Lemma~\ref{lem:udiv} that $u \nmid f$ in $K'[s,u]$.  Therefore, $f = f_1 + uf_2$ with $0 \neq f_1 \in K'[s]$ and $f_2 \in K'[s,u]$.  in particular, $f_1 = \sum_{i,j} c_{ij} s^{pi + qj}$.  
Then $v(f) = w(f) \leq w(f_1) = \min\{pi+qj \mid c_{ij} \neq 0\}$.  As usual, recalling the notation of Lemma \ref{lem:star}, write $\alpha = \sigma(1/f^*) = \frac{X^aY^b}f$, where $(a,b) = (a_1(f), a_2(f))$ and $a\geq 1$, $b \geq 1$. Thus in the sums in (\ref{al:suXY}) above, we have $i \leq a$ and $j \leq b$ for all pairs $(i,j)$ such that $c_{ij} \neq 0$.  Hence, $w(f_1)$ takes the form $pi+qj$ for some $i\leq a$, $j \leq b$.  Therefore $v(\alpha) = pa+qb-v(f) \geq p(a-i) + q(b-j) \geq 0$, and if it is nonzero it must be at least $p$.  Since $R^*$ is generated as a $K$-algebra by all such terms $\sigma(1/f^*)$, it follows that $R^* \subseteq V$.

Now, suppose $\delta \in K[\theta] \setminus (\theta)K[\theta]$. Then by Theorem~\ref{thm:islocalpoly}, $\delta$ is a unit of $R^*$, hence also in $V$, so $v(\delta) = 0$. Hence for any nonnegative integer $m$, we have $v(\theta^m/\delta) = m \cdot (pq+1-h)$.  Thus, it has value 0 $\iff h=pq+1$; otherwise positive.

%Suppose $h < pq+1$; we want to show that $v(\alpha)>0$...

It remains to show that if $h=pq+1$ and $v(\alpha) = 0$, then there exist some $m \geq 0$ and some $\delta \in K[\theta] \setminus (\theta) K[\theta]$ with $\alpha = \theta^m/\delta$, whereas if $h<pq+1$ then $v(\alpha)>0$.
To prove this, let $f$, $\alpha$, $f_1$, $f_2$ be as above.  We proceed by induction on the number $\ell=a+b = a_1(f) + a_2(f)$, noting that the statement is vacuously true for $\ell = 0,1$.

We first dispense with the case that some monomial $c_{ij} s^{pi+qj}$ appearing in $f_1$ satisfies either $i<a$ or $j<b$. Then $v(f) \leq w(f_1) \leq pi+qj \leq pa+qb-p$, so that $v(\alpha) = pa+qb-v(f) \geq p$.

Thus, we may assume that $f_1 = c_{ab} s^{pa+qb}$, so that $c_{ab} \neq 0$. Set $g:=f-c_{ab} X^a Y^b$.  Then rewriting $g$ as an element of $K'[s,u]$, we have $g=g_1 + ug_2$, where $g_1 = f_1 - c_{ab} s^{pa+qb} = 0$.  Hence $u | g$ in $K'[s,t]$, whence $X^q - Y^p | g$ in $K[X,Y]$ by Lemma~\ref{lem:udiv}.  That is, we have $f = c_{ab} X^a Y^b + (X^q - Y^p)^m H$, where $m \geq 1$ and $H \in K[X,Y]$ is relatively prime to each of $X$, $Y$, and $X^q - Y^p$.  Thus $a \geq p$ and $b \geq q$.  Also note that $a_1(H) \leq a-qm$ and $a_2(H) \leq b-pm$. %We may assume $c_{ab} = 1$.

Set $\alpha' := \frac{X^a Y^b} {(Y^p - X^q)^mH}$.  Then there exist nonnegative integers $e_1, e_2$ with $\alpha' = X^{e_1}Y^{e_2} \theta^m \sigma(1/H^*)$.  In particular, $e_1 = a-mp-a_1(H)$ and $e_2 = b-mq-a_2(H)$.  Thus, $c_{ab}\alpha' + 1$ is a unit of $R^*$, so since $\alpha = \frac{\alpha'}{c_{ab}\alpha'+1}$, we have $v(\alpha') = v(\alpha)$, which we assume to be $0$.  But $v(\alpha') = e_1p + e_2q +m(pq+1-h)+ v(\sigma(1/H^*))$, whence since $v(\alpha')=0$, we have $e_1=e_2=0$, and every irreducible factor $\tau$ of $H$ satisfies $v(\sigma(1/\tau^*))=0$.  Moreover if $h<pq+1$ it further follows that $m=0$, so that $f=c_{ab} X^a Y^b$, contradicting the fact that $f$ is relatively prime to $X$ and $Y$, finishing this case.

Then in the remaining case (where $h=pq+1$), by inductive hypothesis each such $\tau$ satisfies $\sigma(1/\tau^*) = \theta^{m(\tau)} / \delta(\tau)$. 
 As these terms are multiplicative, there is some $k \in \N$ and $\epsilon \in K[\theta] \setminus (\theta) K[\theta]$ with $\sigma(1/H^*) = \theta^k/\epsilon$  Thus, we have \[
\alpha = \frac{\alpha'}{c_{ab}\alpha'+1} = \frac{\theta^{m+k}/\epsilon}{c_{ab}(\theta^{m+k}/\epsilon)+1} = \frac{\theta^{m+k}}{c_{ab}\theta^{m+k}+\epsilon}. %= \frac{\theta^{m+ \sum_\tau m(\tau)}}{(\alpha'+1) \cdot \prod_\tau \delta(\tau)}.
\]
Since $c_{ab} \theta^{m+k}+\epsilon \in K[\theta] \setminus (\theta) K[\theta]$, we are done.
%Let $a=a_1(f) = \deg_Xf$ and $b=a_2(f) = \deg_Yf$.  Since $f$ is associate to neither of $X$ and $Y$, there is some pair $(i,j)$ with $c_{ij} \neq 0$ and either $i<a$ or $j<b$.  Thus, $w(f_1) = \min\{pi+qj \mid c_{ij} \neq 0\} \leq \max\{pa+q(b-1), (p-1)a+qb\} = pa+qb-p$.  We have $\sigma(1/f^*) = X^aY^b /f$, so \[
%v(\sigma(1/f^*)) = v\left(\frac {X^a Y^b}f\right) = pa+qb-v(f) \geq pa+qb-(pa+qb-p) = p.
%\]
%Since every element of $R^*$ is a sum of products of elements of $K$ with terms of the form $\sigma(1/f^*)$ for nonconstant irreducible $f\in K[X,Y]$, it follows that $R^* \subseteq V$.
\end{proof}

Recall (see \cite{Za-fc}) that an integral domain is a \emph{finite conductor domain} if the intersection of any pair of principal ideals is finitely generated.

\begin{thm}\label{thm:notfc}
For any $n\geq 2$, the ideal $(1/X_1)R_n \cap (1/X_2)R_n$ is not finitely generated.  Hence $R_n$ is not a finite conductor domain.
\end{thm}

\begin{proof}
Let $n\geq 3$ and suppose $(1/X_1) R_n \cap (1/X_2) R_n = (\alpha_1, \ldots, \alpha_t)R_n$ for some $\alpha_1, \ldots, \alpha_t \in R_n$.  Let $S = R_n[X_3, \ldots, X_n]$.  Let $L = K(X_3, \ldots, X_n)$.  By Lemma~\ref{lem:localizingatonevariable}, we have $S = R(L[X_1, X_2])$.  Let $(-)'$ denote the image of an element of $R_n$ in $S$.  Then $\alpha_j' \in (1/X_1) S \cap (1/X_2) S$ for all $j$, so $(\alpha_1', \ldots, \alpha_t') \subseteq (1/X_1)S \cap (1/X_2) S$.  Conversely let $u \in (1/X_1)S \cap (1/X_2)S$.  Then by clearing denominators, there is some positive integer $d$ such that $(X_3 \cdots X_n)^{-d} u \in (1/X_1) R_n \cap (1/X_2) R_n = (\alpha_1, \ldots, \alpha_t)R_n$.  Since $S = R_n[X_3, \ldots, X_n]$, it follows that $u \in (\alpha_1', \ldots, \alpha_t')S$.  Thus, $(1/X_1) S \cap (1/X_2)S$ is a finitely generated ideal, and we have reduced to the 2-dimensional case.  So from now on we assume $n=2$ and we rewrite $X = X_1$, $Y= X_2$.  For the rest of the proof, we pass to the $R^*$ notation.

Suppose $X R^* \cap YR^* = (\alpha_1, \ldots, \alpha_t)$ for some finite list of nonzero $\alpha_i \in R^*$; a contradiction will complete the proof.  Then there exist $\beta_i, \gamma_i \in R^*$ with $\alpha_i = X \beta_i = Y \gamma_i$ for all $i$.  Write $\gamma_i = c_i + \sum_{j=1}^{m_i} \sigma(1/f_{ij})$ where $c_i \in K$, $m_i \geq 0$, and each $f_{ij} \in D \setminus K$.  If some $c_i \neq 0$, then $\gamma_i$ is a unit by Lemma~\ref{lem:units}, 
so $Y/X = \gamma_i^{-1} \beta_i \in R^*$, which is false by \cite[Example 2.9]{nme-Euclidean}.  Hence $m_i \geq 1$ and $\gamma_i = \sum_{j=1}^{m_i} \sigma(1/f_{ij})$.  Choose some positive integer $q$ that is not a multiple of $\chr K$ and such that $q > \max\{ \deg_X {f_{ij}} \mid 1 \leq i \leq t$, $1 \leq j \leq m_i\}$.  Set $v := v_{1,q,q+1}$ and $\theta = \theta_{1,q} = \frac{X^qY}{Y-X^q}$ as in Lemma~\ref{lem:overp}.

Then $Y \theta = X^q\cdot (\theta + Y) \in XR^*$, so $Y \theta \in XR^* \cap YR^*$. It follows that $\theta \in (\gamma_1, \ldots, \gamma_t)$.  Since $v(\theta)=0$, it follows that for some pair $(i,j)$, we have $v(\sigma(1/f_{ij})) = 0$. By Lemma~\ref{lem:overp}, there exists some positive integer $m$ and some element $\delta \in K[\theta] \setminus (\theta)K[\theta]$ such that $\sigma(1/f_{ij}) = \theta^m \delta^{-1}$.  Write $f=f_{ij}$.

Let $d = \deg_{X}(f)$ and $e = \deg_Y(f)$.  Then by Lemma~\ref{lem:star}, we have $\sigma(1/f) = X^d Y^e / f^*$, where $\deg_X(f^*) \leq d$ and $\deg_Y(f^*) \leq e$.

Write $\delta = c_0 + \sum_{i=1}^s c_i \theta^i$, where each $c_i \in K$ and $c_0 \neq 0$.  Then \[
\frac{X^d Y^e}{f^*} = \sigma(1/f) = \theta^m\delta^{-1} = \frac{(X^qY)^m / (Y-X^q)^m}{c_0 + \displaystyle\sum_{i=1}^s \frac{c_i(X^q Y)^i} {(Y-X^q)^i}}.
\] %\bc could we make the above notation better, maybe with a double fraction \ec
If $m\geq s$, then the latter equation simplifies to an equation where both the numerator and denominator of each fraction is a polynomial, as follows: \[
\frac{X^dY^e}{f^*} = \frac{X^{qm}Y^m}{c_0 (Y-X^q)^m + \sum_{i=1}^s c_i (X^q Y)^i (Y-X^q)^{m-i}}.
\]
Since $q>d$ by the choice of $q$, we have $qm>d$. %\bc why $qm>d$? \ec 
It follows that $X | c_0(Y-X^q)^m$, which contradicts the fact that $c_0 \in K^\times$.

On the other hand if $m<s$, then the equation simplifies with numerators and denominators being polynomials, as follows: \[
\frac{X^d Y^e}{f^*} = \frac{X^{qm}Y^m (Y-X^q)^s}{c_0(Y-X^q)^s + \sum_{i=1}^s c_i (X^q Y)^i (Y-X^q)^{s-i}}.
\]
Cross-multiplying, we have \[
X^{qm}Y^m (Y-X^q)^s f^* = X^d Y^e \cdot (c_0(Y-X^q)^s + \sum_{i=1}^s c_i (X^q Y)^i (Y-X^q)^{s-i}).
\]
Since $qm>d$, it follows that $X | c_0 (Y-X^q)^s$, which again contradicts the fact that $c_0 \in K^\times$.
\end{proof}

Recall %\bc [[CITE]] \ec
that a ring is \emph{coherent} if every finitely generated ideal is finitely presented.  The coherent rings include the Noetherian rings and also all valuation domains (See \cite[Chapter I, \S2, Exercise 12 and Chapter VI, \S1, Exercise 3]{Bour-CA}).%\bc more in general Pr\"{u}fer domains are coherent \ec

\begin{cor}\label{cor:notcoh}
For any $n\geq 2$, the ring $R_n$ is not coherent.  Hence it is non-Noetherian.
\end{cor}

\begin{proof}
This follows from Theorem~\ref{thm:notfc} and \cite[Theorem 2.2]{Ch-dpmod}.
\end{proof}

The next result is notably unlike the behavior of localized polynomial rings.

\begin{thm}\label{thm:notic}
    For any $n\geq 2$, $R_n$ is not integrally closed.
\end{thm}

\begin{proof}
We first consider the 2-dimensional case.  Let $p,q$ be relatively prime integers with $1<p<q$, such that neither $p$ nor $q$ is a multiple of $\chr K$.  By elementary number theory, there is a unique pair of  integers $c,d$ with $qd-pc=1$, $0< c < q$, and $0<d<p$.  
Consider the element $\beta := \beta_{p,q} :=\frac{X^{2q-c}Y^d}{X^q - Y^p} \in F$.  We claim that $\beta$ is integral over $R^*$ -- in fact, $\beta^p \in R^*$ -- but $\beta \notin R^*$.

To see that $\beta^p \in R^*$, simply note the following: \[
\beta^p = \frac{X^{(2q-c)p} Y^{dp}}{(X^q - Y^p)^p} = \left(\frac{X^q Y^p}{X^q - Y^p}\right)^d \cdot \left(\frac{X^q Y^p}{X^q - Y^p} + X^q\right)^{p-d} \cdot X,
\]
which is in $R^*$ since $X \in R^*$ and $\sigma(1/(Y^p - X^q)) = X^q Y^p / (X^q-Y^p) \in R^*$.

On the other hand, let $v = v_{p,q,pq+1}$.  Then \[
v(\beta) = v\left(\frac{X^{2q-c} Y^d}{X^q - Y^p}\right) = (2q-c)p + qd - 2pq=qd-pc=1.
\]

Suppose $\beta \in R^*$.  Since $v(\beta)>0$, it follows that $\beta \in \m_V \cap R^* \subseteq \m$, the maximal ideal of $R^*$.  So by Theorem~\ref{thm:islocalpoly}, we have $\beta = \sum_{i=1}^t \sigma(1/f_i^*)$ for nonconstant polynomials $f_i \in K[X,Y].$  By reordering, let $f_1, \ldots, f_s$ be the polynomials whose only irreducible factor is $X^q - Y^p$ up to associate and multiplicity, whereas each of $f_{s+1}, \ldots, f_t$ has an irreducible factor not associate to $X^q-Y^p$.  Set $\gamma := \sum_{i=1}^s \sigma(1/f_i^*)$ and $\delta := \sum_{i=s+1}^t \sigma(1/f_i^*)$, so that $\beta =\gamma + \delta$.  By Lemma~\ref{lem:overp}, we have $v(\delta)\geq p$, so that since $v(\beta) = 1<p$, we have $v(\gamma) = 1$.

On the other hand, for $1 \leq i \leq s$, there exist $\lambda_i \in K$ and $\ell_i \in \N_0$ with $f_i = \lambda_i (X^q - Y^p)^{\ell_i}$.  Thus, $\sigma(1/f^*) = \lambda_i \theta^{\ell_i}$, so that by Lemma~\ref{lem:overp} we have $v(\sigma(1/f_i^*)) = \ell_i v(\theta) = 0$. %\bc I am confused about the following sentence, I think we should say that $\gamma \in K(\theta)$ and therefore $v(\gamma)=0$ \ec
Thus, either $\gamma=0$ or $v(\gamma)=0$, either of which is a contradiction.  Hence, $\beta \notin R^*$.
%Setting $\alpha = \sigma(1/(X^q-Y^p)^*)$, there is some polynomial $\phi \in K[t]$ such that $\gamma = \phi(\alpha)$. Note that $\phi$ is not the zero polynomial since $v(\gamma) \neq \infty$.  Write $\phi = t^e \psi$, where $e \in \N_0$ and $\psi \in K[t]$ with $\psi(0) \neq 0$.  Then  by Lemma~\ref{lem:unit2}, $\psi(\alpha)$ is a unit of $R^*$, hence also of $V$, so $v(\psi(\alpha)) = 0$.  Also, $v(\alpha)=0$ by Proposition~\ref{pr:overp}.  Thus, we have \[
%q-p = v(\gamma) = v(\alpha^e \psi(\alpha)) = ev(\alpha) + v(\psi(\alpha)) = e\cdot 0 +0=0,
%\]
%which contradicts the fact that $p\neq q$.  Hence, $\beta \notin R^*$.

Finally, we pass to the $n$-dimensional case.  We have $\beta^p \in R_2^* \subseteq R_n^*$.  Since $R_2^* = R_n^* \cap K(X,Y)$ by Proposition~\ref{pr:downdim}, and $\beta \in K(X,Y) \setminus R_2^*$, it follows that $\beta \notin R_n^*$.
\end{proof}

\section{The abundance of prime ideals in $R$}

In this section, we show that $R_n$, which as we have seen is far from Noetherian when $n> 1$ (see Corollary~\ref{cor:notcoh}),  does have infinitely many prime ideals of each height other than $0$ and $n$, a property enjoyed by any $n$-dimensional Noetherian ring, but not by some non-Noetherian rings (e.g. any valuation domain of dimension at least 2).  We start with the following result to bootstrap our efforts.

\begin{prop}\label{pr:infprimes1}
For any $n\geq 2$, $R_n$ has infinitely many 
height one prime ideals.
\end{prop}

\begin{proof}
 In this proof, we use $R^*$ notation.

First suppose $n=2$.  For any relatively prime pair $(p,q)$ of positive integers with $p<q$, let $V, v$, $\theta$, and $h$ be as in Lemma~\ref{lem:overp}, with $h=pq+1$. Let $\p = \p_{p,q}$ be the contraction of $\m_V$ to $R^*$.  Then since $v(\theta) = 0$, we have $\theta \notin \p$.  Since $\p$ is a nonzero prime but not the maximal ideal of $R^*$ (as $\theta \in \m$), it follows that $\p$ is a height one prime.

%\bc in the next paragraph I think we have to replace $(p,q)$ with $(r,s)$ since the valuation $v$ is the one associated to $(p,q)$ \ec
On the other hand, let $(r,s)$ be a different pair of relatively prime positive integers with $r<s$.  We claim that $v(X^s-Y^r)=\min\{v(X^s), v(Y^r)\} = \min\{ps,qr\}$.  Otherwise we would have $v(X^s) = v(Y^r)$, whence $ps=qr$.  But then by assumption of relatively prime pairs, we would have $p=r$ and $q=s$, contradicting the assumption of distinctness. Therefore, $v(\theta_{r,s}) = v(X^s Y^r) - v(X^s-Y^r) = \max\{qr, ps\}$. Thus, $\theta_{r,s} \in \p_{p,q}$.  But by the proof of Lemma~\ref{lem:overp}, $\theta_{r,s} \notin \p_{r,s}$.  Hence, $\p_{p,q} \neq \p_{r,s}$.  Since there are infinitely many such pairs of integers, it follows that $R^*$ has infinitely many height one primes.

Finally, we drop the assumption that $n=2$.  By Proposition~\ref{pr:loclower}, there is a prime ideal $Q$ of $R^*$ such that $R^*_Q$ is isomorphic to the \rco\ of $L[X,Y]$ for some field $L$.  But then by the dimension 2 part of the proof above, $R^*_Q$ has infinitely many height one primes.  Thus, there are infinitely many height one primes of $R^*$ that are contained in $Q$.
\end{proof}

\begin{nota}\label{nota:trivext}
Recall that given a valuation ring with fraction field $K$ and an indeterminate $t$ over $K$, the ring $V(t)$ is a valuation ring  of $K(t)$ called the trivial extension of $V$. %Call $v'$ the valuation associated with $V(t)$.
Given $\phi= \sum_{j=0}^e f_j t^j$ with $f_j \in K$, then the value of $\phi$ with respect to $V(t)$ is $\min_{j} \lbrace  v(f_j)\rbrace$ (See \cite[p. 218]{Gil-MIT}).

By Lemma~\ref{lem:Xiprimes}, there is a prime ideal $Q \in \Spec R_n$ such that $(R_n)_Q = R(K(X_n)[X_1, \ldots, X_{n-1}])$.  Fix this prime for the next two lemmas.
\end{nota}

\begin{lemma}\label{lem:trivext}
Let $V$ be a valuation overring of $R_{n-1}$; then the trivial extension $V(X_n)$ is an overring of $(R_n)_Q$, where $Q$ is as in Notation~\ref{nota:trivext}.
\end{lemma}

\begin{proof}
By the comment before the Lemma, it suffices to show that $1/\phi \in V(X_n)$ for every $\phi \in K(X_n)[X_1, \ldots, X_{n-1}]$.  Since $K(X_n) \subseteq V(X_n)$, we may assume $\phi \in K[X_1, \ldots, X_n]$.  Let $v^\ast$ be the valuation for $V(X_n)$; write $\phi = \sum_{j=0}^e f_j X_n^e$ with $f_j \in K[X_1, \ldots, X_{n-1}]$.  Since $V$ is an overring of $R_{n-1}$, we have that $v(f_j) \leq 0$ whenever $f_j \neq 0$.  Hence, $v^\ast(1/\phi) = -\min \{v(f_j) \mid 0 \leq j \leq e$ and $f_j \neq 0\}\geq 0$.
\end{proof}

\begin{lemma}\label{lem:htineq}
Let $\p \in \Spec R_{n-1}$, and let $V$ be a valuation overring of $R_{n-1}$ centered on $\p$.  Let $\p'$ be the center of $V(X_n)$ in $R_n$.  Then $\hgt \p' \geq \hgt \p$, with equality if $\hgt \p \in \{0,n-2,n-1\}$.
\end{lemma}

\begin{proof}
Let $i=\hgt \p$. If $i=0$, then $\p=(0)$, so that $V = \Frac R_{n-1} = K(X_1, \ldots, X_{n-1})$, whence $\p'=(0)$.  Assume by induction that $i\geq 1$ and the inequality holds for all primes with smaller height.  Note that 
$\p' \cap R_{n-1} = \p$.

Let $\q \subsetneq \p$ with $\q \in \Spec R_{n-1}$ and $\hgt \q = i-1$.  Let $W$ be a valuation overring of $R_{n-1}$ centered on $\q$; let $\q'$ be the center of $W(X_n)$ in $R_n$.  To show that $\q' \subseteq \p'$, it suffices by Proposition~\ref{pr:gensprimes} to show that for any $\phi \in K[X_1, \ldots, X_n]$ with $\frac 1\phi \in \q'$, we have $\frac 1\phi \in \p'$.  Write $\phi = \sum_{j=0}^e f_j X_n^j$, $f_j \in K[X_1, \ldots, X_{n-1}$.  Since $\frac 1\phi \in \q'$, there is some $0 \leq k \leq e$ with $\frac 1{f_k} \in \q$, by the way the valuation on $W(X_n)$ is defined.  Thus, $\frac 1{f_k} \in \p$, so $v^\ast(\frac 1\phi) = -\min\{v(f_j) \mid 0 \leq j \leq e\}\geq -v(f_k) > 0$.  Hence $\q' \subseteq \p'$.  On the other hand $\q' \neq \p'$, since for any $\alpha \in \p \setminus \q$, we have $\alpha \in \p' \setminus \q'$.  Thus, $\q' \subsetneq \p'$, so that \[
\hgt \p' \geq 1 + \hgt \q' \geq 1 + (i-1) = i,
\]
with the second inequality by inductive hypothesis.

Suppose $i = n-1$. Since $X_n \notin \p'$, we have that $\p'$ is not the maximal ideal of $R_n$, so that $\hgt \p' \leq n-1$.  But also $\hgt \p' \geq \hgt \p = n-1$, so that $\hgt \p' = n-1$.

Finally, suppose $i = n-2$.  Let $\m$ be the maximal ideal of $R_{n-1}$.  Since $\m$ contains all nonunits of $R_{n-1}$ and $\hgt \m =n-1$, there is some $\alpha \in \m \setminus \p$.  Thus $\alpha \notin \p'$.  But since $V(X_n) \supseteq (R_n)_Q$ by Lemma~\ref{lem:trivext}, we have $\p' \subseteq Q$.  Moreover, the containment must be strict, since $\alpha \in \m \subseteq Q$ but $\alpha \notin \p'$.  Thus, $\hgt \p' \leq n-2$, so $\hgt \p' = \hgt \p = n-2$. 
\end{proof}

\begin{lemma}\label{lem:pullbackvaluedprimes}
Let $\p$ be a prime ideal of $R_{n-1}$ of height $n-2$. Let $V$ a valuation ring of $R_{n-1}$ centered on $\p$.  Let $\kappa_V$ be the residue field of $V$ and let $\pi: V(X_n) \onto \kappa_V(X_n)$ be the canonical surjection.  Let $W := \pi^{-1}(\kappa_V[X_n^{-1}]_{(X_n^{-1})})$.  Then $W$ is a valuation overring of $R_n$ centered on a prime ideal $\ia$ with $\hgt \ia = n-1$, such that $\ia \cap R_{n-1} = \p$.
\end{lemma}

\begin{proof}
We have that $W$ is a valuation ring with quotient field $K(X_1, \ldots, X_n)$ by \cite[Theorem 2.1(h)]{BaGil-overdiv}.  If $G$ is the value group of $V$, then $G \oplus \Z$, ordered lexicographically, is the value group of $W$%\bc [[CITE]]\ec
.  In particular, given $\phi = \sum_{j=0}^e f_j X_n^j \in K[X_1, \ldots, X_n]$, with each $f_j \in K[X_1, \ldots, X_{n-1}]$, the valuation $w$ is given by $w(\phi) = (v(f_k), -k)$, where $k$ is the largest index $i$ with $0\leq i \leq e$ such that $v(f_i) \leq v(f_j)$ for all $0\leq j \leq e$.  Then $w(1/\phi) = (-v(f_k), k) \geq (0,0)$ since $\frac 1{f_k} \in R_{n-1} \subseteq V$, whence $v(f_k) \leq 0$.

Set $\ia := \m_W \cap R_{n}$ and $\p' := \m_{V(X_n)} \cap R_n$.  By standard pullback results% \bc [[CITE]]
, %\ec
the maximal ideal of $V(X_n)$ is a nonmaximal prime of $W$; thus $\p' \subseteq \ia$.  On the other hand, $X_n^{-1} \in \ia \setminus \p'$, so that $\p' \subsetneq \ia$.  Since $\p$ is a nonmaximal ideal of $R_{n-1}$, there is some nonunit $\alpha$ of $R_{n-1}$ (hence also of $R_n$) that avoids $\p$.  We have $w(\alpha) = (v(\alpha), 0) = (0,0)$, so that $\alpha \notin \ia$.  Thus, $\p' \subsetneq \ia \subsetneq \m_{R_n}$, so that since $\hgt \p' = n-2$ by Lemma~\ref{lem:htineq}, we have $\hgt \ia = n-1$.

Now, $\p = \p' \cap R_{n-1} \subseteq \ia \cap R_{n-1}$.  Hence, $\hgt (\ia \cap R_{n-1}) \geq \hgt \p = n-2$.  But $\alpha \in \m_{R_{n-1}} \setminus \ia$, so $\hgt (\ia \cap R_{n-1}) = n-2$, whence $\ia \cap R_{n-1} = \p$.
\end{proof}

\begin{thm}\label{thm:infprimesallheights}
For every $1\leq i \leq n-1$, there exist infinitely many primes of $R_n$ of height $i$.
\end{thm}

\begin{proof}
When $n=0,1$, the statement is vacuous.  Moreover, since when $n\geq 2$ we know that $R_n$ has infinitely many height one primes by Proposition~\ref{pr:infprimes1}, the result holds for $n=2$.  Thus, we assume inductively that $n>2$ and the result holds for smaller $n$.  Since $(R_n)_Q = R(K(X_n)[X_1, \ldots, X_{n-1}])$ (see Notation~\ref{nota:trivext}), it has infinitely many primes of height $i$ for $1\leq i \leq n-2$, which then restrict to distinct primes of these heights in $R_n$ via the localization map.  So we need only show that $R_n$ has infintely many primes of height $n-1$.

Let $\p$, $\q$ be distinct prime ideals of height $n-2$ in $R_{n-1}$.  By Lemma~\ref{lem:pullbackvaluedprimes}, there are valuation overrings $W_1$, $W_2$ of $R_n$ whose centers in $R_n$ are height $n-1$ primes $\p'$, $\q'$ such that $\p' \cap R_{n-1} = \p$ and $\q' \cap R_{n-1} = \q$.  Since $\p\neq \q$, it follows that $\p' \neq \q'$.  Since there are infinitely many primes of height $n-2$ in $R_{n-1}$ by inductive hypothesis, it thus follows that there are infinitely many primes of height $n-1$ in $R_n$.
\end{proof}

\section{The dimension 2 case}
In this section, we work in 2 variables, so that $D = K[X,Y]$ where  $X=X_1$, $Y=X_2$ for short, $R = R(K[X,Y])$, $F = K(X,Y)$, etc.  We have done likewise in many results earlier in the paper in service of extending the results to higher dimensions. However, for each of the results in this section, either we do not know how to extend it into higher dimension, or else we know it to be false in higher dimension.  In the dimension 2 case, we will show that $R$ has a non finitely generated integral overring, that localizing $R$ at height one primes always yields Noetherian domains, and that any finitely generated proper ideal lives in almost all height one primes.

We start by expanding Theorem~\ref{thm:notic} to show that the integral closure if $R^*$ is quite a bit larger than $R^*$ itself:

\begin{prop}\label{pr:infiniteoverring}
There is an overring $S$ of $R=R_2$, that is integral over $R$ but not finitely generated over it.
\end{prop}

\begin{proof}
As usual, we will work with $R^*$ instead of $R$.

Let $\Sigma := \{\beta_{p,q}\}$ as in the proof of Theorem~\ref{thm:notic}, where the pairs $(p,q)$ range over all relatively prime pairs of integers $1<p<q$ such that neither $p$ nor $q$ is a multiple of $\chr K$.  Let $S := R^*[\Sigma]$.  Then as seen in the proof of Theorem~\ref{thm:notic}, each $\beta_{p,q}$ is in the integral closure of $R^*$. Hence, $S$ is integral over $R^*$.  

Suppose $S$ is finitely generated as an $R^*$-algebra.  Then there is a finite list of such pairs $\{(p_i,q_i)\}_{1 \leq i \leq s}$ such that $S = R^*[\beta_{p_1,q_1}, \ldots, \beta_{p_s,q_s}]$.  Choose $r>\max \{p_i \mid 1\leq i \leq s\}$ such that $r>1$.  Then $(r,r+1)$ is such a pair, so $\beta_{r,r+1} \in S$.  Let $v = v_{r,r+1, r^2 + r+1}$.

Then for any relatively prime pair $(p,q)$ with $p<r$, we claim that $v(X^q-Y^p) = \min\{v(X^q), v(Y^p)\}$.  If this were not the case, we would have $rq=v(X^q) = v(Y^p)=(r+1)p$, so that $(q-p)r = p$, contradicting the facts that $q-p\geq 1$ and $r>p$.  Thus, $v(X^q-Y^p) = \min\{rq, (r+1)p\}$.

Now choose any $(p,q) = (p_i,q_i)$ with $1 \leq i \leq s$.  Let $(c,d)$ be the unique pair of integers with $0<c<q$, $0<d<p$, and $qd=pc+1$.  Then \begin{align*}
v(\beta_{p,q}) &= v\left(\frac{X^{2q-c}Y^d}{X^q - Y^p}\right) = (2q-c)r+ d(r+1) - \min\{rq, (r+1)p\}\\
&\geq (2q-c)r+d(r+1)-rq = (q-c)r+(r+1)d \geq 2r+1.
\end{align*}
On the other hand, by the proof of Theorem~\ref{thm:notic}, we have $v(\beta_{r,r+1}) = 1$.

Let $Z_1, \ldots, Z_s$ be algebraically independent indeterminates over $R^*$. It follows that if $g \in R^*[Z_1, \ldots, Z_s]$ %\bc we did not use the notation $Z_1, \ldots, Z_s$ before \ec
such that $\beta_{r,r+1} = g(\beta_{p_1,q_1}, \ldots, \beta_{p_s,q_s})$, then $g$ has a nonzero constant term $c$, and $1=v(\beta_{r,r+1}) =v(c)$.  But this contradicts the fact (see Lemma~\ref{lem:overp}) that every element of $R^*$ has value either $0$ or $\geq r$ under $v$.  Thus, $S$ is not finitely generated as an $R^*$-algebra.
\end{proof}

Our main result for dimension 2 shows that the localizations at height one primes are surprisingly well-behaved.  First, though, we need the following lemma.

\begin{lemma}
\label{lem:Xinallprimes2}
Let $\p \in \Spec R_n$ and $1 \leq i \leq n$ such that $\p \nsubseteq \q_i$, where $\q_i$ is as in Lemma~\ref{lem:Xiprimes}.  Then $1/X_i\in \p$.

The $\q_i$ are mutually incomparable, and for each $i$, $1/X_j \in \q_i$ for each $j \neq i$.

If $n =2$, then any nonzero prime distinct from $\q_1$ and $\q_2$ contains $(1/X_1, 1/X_2)R$. Hence with the notation of Lemma~\ref{Xiprimes}, $\p_1=\q_2$ and $\p_2=\q_1$.
\end{lemma}

\begin{proof}
By Lemma~\ref{lem:Xiprimes}, any prime ideal avoiding $X_i^{-1}$ must be contained in $\q_i$. Thus $1/X_i \in \p$.

Now let $i,j$ be distinct integers between $1$ and $n$.  Combining Lemmas~\ref{lem:localizingatonevariable} and \ref{lem:Xiprimes}, and Theorem~\ref{thm:dim} yields $\hgt \q_i = \hgt \q_j = n-1$.  Thus, $\q_i$, $\q_j$ must be incomparable.  Since $\q_i \nsubseteq \q_j$, we have $X_j \in \q_i$ by the first paragraph.

In the $n=2$ case, by Theorem~\ref{thm:dim} we have that any nonmaximal nonzero prime has height 1.  In particular, if $\p$ is distinct from $\q_1$ and $\q_2$, then since $\hgt \p = 1 = \hgt \q_1 = \hgt \q_2$, we have $\p \nsubseteq \q_i$ for $i=1,2$.  Since $1/X_i \in \p_i \setminus \q_i$ for $i=1,2$, the final claim follows.
\end{proof}

\begin{thm}
\label{localizations}
Let $\p$ be a height one prime ideal of $R=R_2$%, where $R = R(D)$ and $D = K[X,Y]$
. Then $R_\p$ is a Noetherian one-dimensional local domain.
%\begin{enumerate}
%\item[(1)] $R^*_{\p}$ is a Noetherian one dimensional local domain.
%\item[(2)] 
%\end{enumerate}
\end{thm}

\begin{proof}
By Lemma~\ref{Xiprimes}, we may assume $\p \neq \p_1, \p_2$.  We work in $R^*$ and use the notation $D=K[X,Y]$, $R^* = \sigma(D)$.  Since $\p$ is not maximal there must exist $f \in (X,Y)D$ irreducible and not associate to either  $X$ or $Y$ in $D$, such that $\alpha:= \sigma(1/f^*) \not \in \p$.
As neither $\alpha$ nor $\alpha^{-1}$ are in $S= D_{(X,Y)}$ it follows from \cite[Theorem 7]{Sei-notedim} that $S[\alpha]$ is a two dimensional Noetherian ring having a height one prime ideal $\q$ generated by $X,Y$. Hence $S[\alpha]_{\q}= D[\alpha]_{(X,Y)}$ is a one dimensional local Noetherian domain.

Moreover, since $S \subseteq R^*$, we get $S[\alpha] \subseteq R^*$. 
 Let us show that $\p \cap S[\alpha] = \q$. By Lemma~\ref{lem:Xinallprimes2}, $X, Y \in \p$, so that $\q \subseteq \p \cap S[\alpha]$.  For the reverse containment, it will be enough to show that $\p \cap S[\alpha]$ is a height one prime, for which it will suffice to show it is not a maximal ideal.

Every maximal ideal of $S[\alpha]$ that contains $\q$ is of the form $(\q, h(\alpha))$ where $h \in K[T]$ is an irreducible monic polynomial. This is because by \cite[Theorem 7]{Sei-notedim}, we have $S[\alpha]/\q \cong K[T]$, where $T$ is an indeterminate over $K$ and $\bar \alpha \mapsto T$ in the isomorphism. %\bc here $T$ is both an indeterminate and a domain \ec 
We know that $\alpha \not \in \p$. If $h(T) \neq T$, then it has a nonzero constant term, so by Theorem~\ref{thm:islocalpoly}, $h(\alpha)$ is a unit in $R^*$, whence $h(\alpha) \notin \p$. On the other hand, $\alpha \notin \p$ by assumption. 

Thus  $\p \cap S[\alpha] = \q$, so that $S[\alpha]_{\q} \subseteq R^*_{\p}$. By the Krull-Akizuki Theorem \cite[Theorem 11.7]{Mats}, $R^*_{\p}$ is Noetherian, finishing the proof.
\end{proof}

\begin{rem}\label{rem:localizationDVR}
Suppose, in the setting of Theorem~\ref{localizations}, that there exists $f \in (X,Y)D \setminus (X,Y)^2D$ such that $1/f^* \not \in \p$. Then $R_{\p}$ is not merely one-dimensional and Noetherian, but a DVR.

To see this, and continuing the notation in the proof above, let $Z$ be an indeterminate over $D$ and let $\pi: K[X,Y,Z] \to D[\alpha]$ be the unique $K$-algebra homomorphism that fixes $D$ and sends $Z \mapsto \alpha$. Note that $\pi$ is surjective. Since $D[\alpha]$ has dimension 2, the kernel of $\pi$ is a height one prime of $K[X,Y,Z]$. It is clear that the polynomial $h= fZ-X^{a_1(f)}Y^{a_2(f)}$ is irreducible in $K[X,Y,Z]$ and contained in the kernel of $\pi$. It follows that
$$ D[\alpha] \cong \frac{K[X,Y,Z]}{(h)}.  $$ The ring $S[\alpha]_{\q} \cong \left(\frac{K[X,Y,Z]}{(h)}\right)_{(X,Y)}$ is a DVR if and only if $h$ is a regular parameter in $K[X,Y,Z]_{(X,Y)}$. This happens if and only if $f \in (X,Y)D \setminus (X,Y)^2D$.
In case $S[\alpha]_{\q}$ is a DVR, we clearly have $S[\alpha]_{\q} =  R^*_{\p}$ since a DVR has no proper  overring other than its fraction field.
\end{rem}

Next we show a ``cofinite character'' result that is somewhat dual to the finite character property of Krull domains.

\begin{thm}
 \label{allbutfinitelymnyprimes}% Let $R^* = \sigma(R(D))$, where $D = K[X,Y]$.
 Every finitely generated proper ideal of $R=R_2$ is contained in all but finitely many prime ideals.
 \end{thm}
 
 \begin{proof}
Every nonunit element $\phi \in R^*$ can be written as a finite sum $\phi = \phi_1 + \ldots + \phi_t$ such that any $\phi_j$ is a finite product of elements of the form $\sigma(\frac{1}{f})$ with $f$ irreducible in $K[X,Y]$. Thus it is sufficient to prove that any $\sigma(\frac{1}{f})$ with $f$ irreducible is contained in all but finitely many primes of $R^*$. Since we already know this fact for $X$ and $Y$ by Lemma~\ref{lem:Xinallprimes}, we can assume $f$ is not an associate of $X$ nor $Y$. Set $\alpha= \sigma(\frac{1}{f})$. We know that $\alpha$ is in the maximal ideal of $R^*$. By the proof of Theorem \ref{localizations}, we get that if a prime $\p$ of $R^*$ does not contain $\alpha$, then $R^*_{\p}$ contains the one dimensional Noetherian local domain $D[\alpha]_{(X,Y)}$. Suppose there exist two distinct nonzero prime ideals $\p$, $\q$ of $R^*$ with $\alpha \notin \p \cup \q$. Necessarily $\p$ and $\q$ have height one. We show that $R^*_{\p}$ and $R^*_{\q}$ cannot be contained in a common valuation overring. Suppose by way of contradiction that $R^*_{\p} \cup R^*_{\q} \subseteq V$ for some valuation ring $V$ contained in $K(X,Y)$. Then $V$ is an overring of $D[\alpha]_{(X,Y)}$, hence a DVR. Since $R^*_{\p}$ and $R^*_{\q}$ are one dimensional, the maximal ideal $\m_V$ of $V$ contains both the maximal ideals of $R^*_{\p}$ and $R^*_{\q}$. Therefore the intersection $ R^*_{\p} \cap R^*_{\q} $ is local. For this observe that given two nonunits $\beta, \theta \in R^*_{\p} \cap R^*_{\q}$, we have $\beta, \theta \in \p R^*_{\p} \cup \q R^*_{\q} \subseteq \m_V$ and hence their sum $\beta + \theta \in \m_V \cap R^*_{\p} \cap R^*_{\q} \subseteq \p R^*_{\p} \cap \q R^*_{\q}$ is a nonunit in $ R^*_{\p} \cap R^*_{\q} $. Moreover, we have $R^* \subseteq  R^*_{\p} \cap R^*_{\q} \subseteq R^*_{\p}.$ Since $\alpha^{-1} \in (R^*_{\p} \cap R^*_{\q}) \setminus R^*$, the maximal ideal of $R^*_{\p} \cap R^*_{\q}$ must contract to a nonzero prime ideal of $R^*$ not containing $\alpha$. But any such ideal has height one, so that $R^*_{\p} \cap R^*_{\q}$ contains the localization of $R^*$ at some height one prime. This is a contradiction since 
 two localizations at distinct height one primes cannot be comparable with respect to inclusion.
 Hence, $R^*_{\p}$ and $R^*_{\q}$ cannot have a common valuation overring.
 
 Now, suppose there are infinitely many distinct primes of $R^*$ not containing $\alpha$. Then the above implies that $D[\alpha]_{(X,Y)}$ has infinitely many valuation overrings, but it is clearly a contradiction since $D[\alpha]_{(X,Y)}$ is local, Noetherian and one dimensional.   
 \end{proof}

Recall that in a domain $S$, with fraction field $F$, for any $S$-submodule $I$ of $F$ we can write $I^{-1} := \{x \in F \mid xI \subseteq S\}$.  Then for any ideal $I$ of $S$, we set $I_v := (I^{-1})^{-1}$ and $I_t := \bigcup \{J_v \mid J \subseteq I$ and $J$ is finitely generated$\}.$  A \emph{t-ideal} is then an ideal $I$ such that $I=I_t$. If there is a unique maximal element among the proper t-ideals of $S$, we say $S$ is \emph{t-local}. See \cite{FoZa-tlocal}.%\bc here we may add the reference to the paper of Fontana and Zafrullah  ``$t$-local domains'', where they explain how they have common features with valuation domains \ec

\begin{cor}\label{cor:tlocal}
The ring $R=R_2$ is $t$-local. % That is, there is a unique maximal t-ideal.
\end{cor} 
 
\begin{proof}
In fact the unique maximal ideal $\m$ of $R$ is a $t$-ideal.  To see this, let $I$ be a finitely generated proper ideal of $R$.  Then by Theorem~\ref{allbutfinitelymnyprimes}, $I$ is contained in some height one prime ideal $\p$.  But $\p$ is a $t$-ideal by %\bf kang or maybe gilmer's book \rm
\cite[top of p. 23]{El-clbook}. Hence, $I_t \subseteq \p_t = \p \subset \m$.  Thus, $\m_t = \m$.
\end{proof}

\section{Questions}
The study of \rcos\ is an entirely new field of inquiry.  There are many interesting questions one could ask about this particular $R$, or about \rcos\ in general.  The following are just some questions that occurred to these authors, but such questions are easy to generate.  As seen below, some of these questions have had progress on them since they were first proposed in an earlier draft of this paper.
\begin{question}
What can be said about the integral closure of $R$, where $R$ is the reciprocal complement of a polynomial ring in 2 or more variables over a field?  In particular: \begin{enumerate}[(a)]
 \item Is the integral closure of $R$ infinitely generated over $R$?

 We can't conclude this from Proposition~\ref{pr:infiniteoverring} since $R$ is not a Noetherian ring, so finitely generated $R$-modules can have infinitely generated submodules.
 \item Is the integral closure of $R$ local? If not, is it at least semilocal?
 \item Is the integral closure of $R$ completely integrally closed?
 \end{enumerate}
\end{question}

\begin{question}
Let $R$ be the reciprocal complement of a polynomial ring in finitely many variables over a field. Is  $R$ a \emph{strong B\'ezout intersection domain} (SBID)  (see \cite{GuLo-Bezout})?  That is, is it true that every finite intersection of non-comparable principal ideals fails to be finitely generated?
\end{question}

Some positive evidence is given by Theorem~\ref{thm:notfc}, and also by Corollary~\ref{cor:tlocal}, since any SBID is t-local.%. \bc also some evidence is given by the fact that $R^*$ is $t$-local, which is a necessary condition to be SBID \ec

\begin{question}
    Let $D$ be a Noetherian domain.  Is $R(D)$ a G-domain?
\end{question}
For some evidence of this, see Proposition~\ref{pr:Gdom}.  More generally, by \cite[Theorem 2.12]{Gu-recomp}, the above holds whenever $\dim R(D) <\infty$.

\begin{question}
    Let $D$ be an integral domain of dimension $\geq 2$.  Assume that any nonzero Egyptian element of $D$ is a unit.  Must $R(D)$ be non-Noetherian?
\end{question}

We have seen in Corollary~\ref{cor:notcoh} that $D = D_n$ is an example of the above phenomenon when $n\geq 2$.  More generally, by \cite[Corollary 2.9]{Gu-recomp}, the answer is yes whenever $\dim R(D) \geq 2$.

\begin{question}
    For any integral domain $D$, must we have $\dim R(D) \leq \dim(D)$?
\end{question}

Note that there is no hope for \emph{equality} in the above, as we have for $D=D_n$ by Theorem~\ref{thm:dim}.  Indeed, the quantity $\phi(D) := \dim(D) - \dim R(D)$ can be any nonnegative value $\omega$, by letting $A$ be a Jaffard (e.g. Noetherian) Egyptian domain of dimension $\omega$ and $D = A[X_1, \ldots, X_d]$; then by Proposition~\ref{pr:reductiontoKalgebra} and Theorem~\ref{thm:dim}, $\dim R(D) = d$, but $\dim D = d+\dim A$, so $\phi(D) = \dim A=\omega$.  Moreover, one can make $A$ have any dimension $\omega$ by letting $G = \Z^{\oplus \omega}$ and $A = K[G]$ for any field $K$, which is Egyptian by \cite[Proposition 3]{GLO-Egypt}.

By \cite[Theorem 5.5]{Gu-recomp}, for any nonnegative integer $c$, one can construct integral domains $D$ all of whose Egyptian elements are units, such that $\phi(D)=c$.

The above question has a positive answer when $D$ is finitely generated over a field or falls into certain classes of semigroup algebras \cite[Theorem 3.2, Remark 4.11]{Gu-recomp}.

\begin{question}
    Let $D$ be a Noetherian integral domain with $\dim D = n\geq 2$. Are there infinitely many prime ideals of $R(D)$ of height $i$ for each $1\leq i \leq n-1$?
\end{question}

We see an example of this phenomenon when $D=D_n$ by Theorem~\ref{thm:infprimesallheights}.  If $D$ is not restricted to be Noetherian, however, there are counterexamples \cite[Theorem 4.2]{Gu-recomp}.

\section*{Acknowledgment}
The authors would like to thank the anonymous referees for their thorough work on our paper, which reads much better and has clearer focus as a result.

\newcommand{\etalchar}[1]{$^{#1}$}
\providecommand{\bysame}{\leavevmode\hbox to3em{\hrulefill}\thinspace}
\providecommand{\MR}{\relax\ifhmode\unskip\space\fi MR }
\providecommand{\MRhref}[2]{%
  \href{http://www.ams.org/mathscinet-getitem?mr=#1}{#2}
}
\providecommand{\href}[2]{#2}

\end{document}